\documentclass[11pt]{article}

\usepackage{graphicx}
\usepackage{color}
\usepackage{amssymb,latexsym,amsmath}
\usepackage{epstopdf}

\usepackage[T1]{fontenc}
\usepackage{hyperref}
\usepackage{url}

\newtheorem{lemma}{Lemma}
\newtheorem{theorem}{Theorem}
\newtheorem{proposition}{Proposition}

\newtheorem{problem}{Problem}
\newtheorem{corollary}{Corollary}
\newtheorem{remark}{Remark}

\newcommand{\G}{\Gamma}
\newcommand{\B}{\beta}
\newcommand{\g}{\gamma}
\newcommand{\proof}{\noindent{\em Proof: }}

\newcommand{\forme}[1]{}

\def\wbull{\hfill\vrule height .9ex width .8ex depth -.1ex}

\begin{document}

\date{}
\title{Unique identification and domination of edges in a graph: The vertex-edge dominant edge metric dimension}
\author{{\bf Hafiz Muhammad Ikhlaq}\\
Department of Mathematics,\\
COMSATS University Islamabad (CUI), Lahore Campus,\\
Lahore, 54000, Pakistan\\
\\
{\bf Sakander Hayat}\\
School of Mathematics and Information Sciences,\\
Guangzhou University, Guangzhou,\\
510006, Guangdong, P.R. China\\
e-mail: sakander1566@gmail.com\\
\\
{\bf Hafiz Muhammad Afzal Siddiqui}\\
Department of Mathematics,\\
COMSATS University Islamabad (CUI), Lahore Campus,\\
Lahore, 54000, Pakistan\\
e-mail: hmasiddiqui@gmail.com\\
}
\maketitle

\begin{abstract}
Dominating sets and resolving sets have important applications in control theory and computer science.
In this paper, we introduce an edge-analog of the classical dominant metric dimension of graphs.
By combining the concepts of a vertex-edge dominating set and an edge resolving set, we
introduce the notion of a vertex-edge dominant edge resolving set of a graph. We call the minimum cardinality
of such a set in a graph $\G$, the vertex-edge dominant edge metric dimension $\g_{emd}(\G)$ of $\G$.
The new parameter $\g_{emd}$ is calculated for some common families such as paths, cycles,
complete bipartite graphs, wheel and fan graphs. We also calculate $\g_{emd}$ for some Cartesian
products of path with path and path with cycle. Importantly, some general results and bounds are presented for this new parameter.
We also conduct a comparative analysis of $\g_{emd}$ with the dominant metric dimension of graphs.
Comparison shows that these two parameters are not comparable, in general.
Upon considering the class of bipartite graphs, we show that
$\g_{emd}(T_n)$ of a tree $T_n$ is always less than or equal to its dominant metric dimension.
However, we show that for non-tree bipartite graphs, the parameter is not comparable just like general graphs.
Based on the results in this paper, we propose some open problems at the end.
\end{abstract}
\centerline{{\bf 2020 Mathematics Subject Classification:} 05C12, 05C69}
\begin{quote}
\end{quote}

\section{Introduction}
Distance based graph-theoretic parameters have a long history. In particular, every connected graph $\G$ is a metric space
i.e. $(V(\G),d)$, where the distance function $d:=d(x,y)$ is the length of a shortest path between $x$ and $y$. See more about this
direction in the book by Deza \& Laurent \cite{DL97}. From application perspective, Harold Wiener \cite{W47} in 1947 defined
path number to be the sum of distances over all the unordered pairs of vertices and showed that it correlates well with the
boiling points of alkanes. See a recent survey by Goddard \& Oellermann \cite{GO2011} on distance in graphs.

Independently, Slater \cite{sla} in 1975 and Harary \& Melter \cite{har} in 1976 introduced a distance related parameter known
as the metric dimension of graphs. The metric dimension is the cardinality of a minimum resolving set in a graph. Resolving sets
and metric dimension have numerous applications in various scientific disciplines such as network discovery \& verification \cite{bee},
mastermind game \cite{cac}, robot navigation \cite{khu}, coin weighing problems \cite{seb}, combinatorial search and optimization \cite{seb},
pattern recognition \& image processing \cite{Mel}, among others. The success of the metric dimension motivated researcher to propose other
resolvability related parameters such as independent resolving set and its minimality by Chartrand et al. \cite{char}, strong resolving set
in graph and digraphs and its minimality by Oellermann \& Peters-Fransen \cite{oel}, local resolving set and metric dimension by
Okamoto et al. \cite{oka}, among others. For more mathematical results regarding resolvability, we refer the reader to a survey by Chartrand \& Zhang \cite{cha}.

Dominating sets and the domination number of graphs is another important research direction of graph theory which has a rich history.
Berge \cite{BER} in 1958 gave the idea of dominating sets and their applications. By Berge \cite{BER}, the domination number
was referred as the ``coefficient of external stability'' in network analysis. The name dominating set was assigned by Ore \cite{Ore1960}
in 1962. In 1977, Cockayne \& Hedetniemi \cite{COC} made an interesting and extensive survey of the results known at that time about dominating sets
in graphs. Rao \& Sreenivansan \cite{RAO} studied the split domination number of arithmetic graphs.
For the detailed mathematical treatment, see the book on domination theory of graph by Haynes et al. \cite{HHS1998}.

Brigham et al. \cite{bri} combined the concepts of a resolving and a dominating set to propose
resolving dominating set of a graph. Cardinality of a minimum such set is known as the dominant metric dimension $\g_{md}(\G)$.
Brigham et al. \cite{bri} calculated the dominant metric dimension of some classical families such as paths, cycles,
complete bipartite etc. Moreover, they proved some general results such as
$$\max \{ \B(\G), \g(\G)\} \leq \g_{md}(\G) \leq \B(\G) + \g(\G).$$
Henning \&  Oellermann \cite{hen}, later on, combined the concepts of a locating dominating set and a resolving set.
The cardinality of such a minimum set in called the metric locating-dominating number $\g_{mld}(\G)$ of a graph $\G$.
Henning \& Oellermann \cite{hen} showed $\g_{mld}(\G)$ satisfy $\g(\G) \leq \g_{mld}(\G) \leq n-1$ for every $n$-vertex
graph $\G$. Gonzalez et al. \cite{GON} further studied $\g_{mld}(\G)$ and showed that it satisfies
$$\max\{ \B(\G),\g(\G)\}\leq \g_{mld}(\G)\leq\B(\G)+\g(\G).$$
Susilowati et al. \cite{SUS} showed that the lower bound on $\g_{mld}(\G)$ by Henning \& Oellermann is not sharp,
whereas,  the upper bound on $\g_{md}(\G)$ by Brigham et al. \cite{bri} and the upper bound on $\g_{mld}(\G)$ by
Gonzalez et al. \cite{GON} are not sharp.

Kelenc et al. \cite{kele} proposed the idea of a edge resolving set i.e. resolvability of edges in a graph by its vertices.
The corresponding parameter is known as the edge metric dimension $\B_e(\G)$ of a graph $\G$. They calculated $\B_e(\G)$ for
classical families and provided comparison between the metric dimension and edge metric dimension of graphs and showed that they
are not comparable, in general. Such variations has attracted attention of several researches, see for instance,
\cite{Huang,Knor,Zhang,Zhu}.

Boutrig et al. \cite{bou} introduced an edge-analog of the dominating set i.e. vertices which dominate edges of a graph.
They called such a set a vertex-edge dominating set and the cardinality of such a minimum set is called the vertex-edge
domination number. Lewis \cite{Lew} in his PhD thesis studied the vertex-edge domination in detail. Lewis et al. \cite{Lew1}
provided comparison of domination number and the vertex-edge domination number of graphs.
Peters \cite{Pet} in his PhD thesis provided a detailed study on the vertex-domination in graphs and provided
several upper and lower bounds on the vertex-edge domination number. Thakkar \& Jamvecha \cite{tha} further studied
the vertex-edge domination in graphs.

In this paper, we propose to study vertex-edge dominant edge resolving sets in graphs. The minimum cardinality of such a set in a graph $\G$ is called
the vertex-edge dominant edge metric dimension of $\G$. In other words, we introduce the edge version of the
dominant metric dimension of graphs. This paper is organized as follows: In Chapter \ref{prelim}, we introduce
all the necessary definitions and preliminary results which are used in subsequent sections. In Chapter \ref{sec3},
we introduce the concept of a vertex-edge dominant edge resolving set and the vertex-edge dominant edge metric
dimension $\g_{emd}(\G)$ of a graph $\G$. We calculate $\g_{emd}$ for classical families such as paths, cycles,
complete graphs, complete bipartite graphs, wheel and fan graphs. Chapter \ref{sec4} computes $\g_{emd}$ for some
Cartesian products such as path vs. path and path vs. cycle. Chapter \ref{sec5} presents some general results and
bounds on $\g_{emd}$. Chapter \ref{sec6} draws a comparison between the dominant metric dimension and its edge
version introduced in this paper. Although, these two parameters are not comparable in general, for trees specifically,
we show that $\g_{emd}(T_n)$ of a tree $T_n$ is always less than or equal to its dominant metric dimension.
Chapter \ref{conc} concludes the findings and propose some open problems which naturally arise from this study.

\section{Preliminaries}\label{prelim}

All graphs in this paper are finite, simple, undirected and connected. A simple graph $\G$ is a pair $\G=(V,E)$ consists of a vertex set $V$ and an edge set $E\subseteq{V\choose 2}$.
The cardinality $\mid V\mid$ (resp. $\mid E\mid$) of $V$ (resp. $E$) is called the order (resp. size) of $\G$.
For a vertex $x\in V$, the open (resp. closed) neighborhood of $x$ is defined as $N_{\G}(x)=\{y\in V: xy\in E\}$
(resp. $N_{\G}[x]=N_{\G}(x)\cup\{x\}$). For $x,y\in V$, let $\ell\big(p(x,y)\big)$ be the length of a path $p(x,y)$
between $x$ \& $y$. The distance between $x,y\in V$ is defined to be the length of a shortest path
between $x$ \& $y$ i.e. $d_{\G}(x,y)=\min\{\ell\big(p(x,y)\big):x,y\in V\}$. For a vertex $w\in V$
and an edge $e=xy$, the distance between $w$ and $e$ is defined to be $d_\G(w,e)=\min\{d_{\G}(w,x),d_{\G}(w,y)\}$. we usually omit $\G$ from $N_{\G}(x)$, $N_{\G}[x]$,
and $d_{\G}(x,y)$ and write $N(x)$, $N[x]$, and $d(x,y)$ instead. The path (resp. cycle) graph
on $n$ vertices is usually denoted by $P_n$ (resp. $C_n$). Similarly, the complete graph of order $n$
is denoted by $K_n$. The star graph of order $n+1$ is denoted by $S_{1,n}$. Moreover, the wheel and
fan graph on $n+1$ vertices are denoted by $W_{1,n}$ and $F_{1,n}$, respectively. Similarly, the complete
bipartite graph is denoted by $K_{n,m}$.

For a graph $\G$, a subset $S \subseteq V$ is said to be a dominating set if for every $x \in V \setminus S$,
we have $N[x]\cap S\neq\emptyset.$  Minimum cardinality of such a set in $\G$ is knows as the dominating number
usually denoted by $\g(\G)$. A vertex $w\in V$ is said to resolve a pair of vertices $x,y\in V$, if $d(w,x)\neq d(w,y)$.
Let $B=\{x_1, x_2,\ldots,x_p\}\subset V$ be an ordered subset. The representation/code $r(y|B)$ of a vertex $y$ with respect to $B$ is
the distance vector $r(y|B)=\big(d(y, x_1), d(y, x_2),\ldots, d(y, x_p)\big)$. Such a subset $B$ is said to be
a resolving set of $\G$ if every $y\in V\setminus B$ has a distinct code corresponding to $B$. Minimal cardinality
of such a resolving set in $\G$ is called the metric dimension of $\G$ usually denoted by $\B(\G)$.

Combining the concepts of a dominating and a resolving set in a graph yields a dominant resolving set.
Formally, an ordered subset $D \subset V$ in $\G$ is said to be a dominant resolving set, if $D$ is simultaneously
a dominating set and a resolving set. Minimum cardinality of such as set is called the dominant metric dimension
of $\G$ and is denoted by $\g_{md}(\G)$. A dominant resolving set of cardinality $\g_{md}(\G)$ in $\G$ is called a
dominant basis of $\G$. For classical families of graphs, Susilowati et al. \cite{SUS} showed the following result.
\begin{theorem}\emph{\cite{SUS}}
The dominant metric dimension $\g_{md}$ of some classical families are calculated as follows:
\begin{itemize}
\item[\emph{(i)}] For $n>4$, we have $\g_{md}(P_n)= \g(P_n)= \lceil \frac{n}{3} \rceil$,
\item[\emph{(ii)}] For $n\geq7$, we have $\g_{md}(C_n)= \g(C_n)= \lceil \frac{n}{3} \rceil$,
\item[\emph{(iii)}] For $n\geq2$, we have $\g_{md}(S_{1,n})=n$,
\item[\emph{(iv)}] For $n\geq2$, we have $\g_{md}(K_n)= \B(K_n)= n-1$,
\item[\emph{(v)}] For $n,m\geq2$, we have $\g_{md}(K_{n,m})= \B(K_{n,m})= n+m-2.$
\end{itemize}
\end{theorem}

Kelenc et al. \cite{kele} introduced an edge version of the metric dimension of graphs.
The vertex $x$ is said to resolves the edges $e, f\in E$, if $d(e,x)\neq d(f,y)$.
For an ordered subset $A=\{x_1, x_2,\ldots, x_q\}\subset V$, the representation/code $r(e|A)$ of
an edge $e$ with respect to $A$ is the edge distance vector $r(e|A)=\big(d(e,x_1),d(e,x_2),\ldots,d(e,x_q)\big)$.
For such an $A$, if every edge of $\G$ has a different code, then $A$ is said to be an edge-resolving set of $\G$.
Denoted by $\B_e(\G)$, the minimum cardinality of an edge resolving set in $\G$ is called the edge metric
dimension of $\G$. An edge resolving set of cardinality $\B_e(\G)$ in $\G$ is known as the edge metric
basis of $\G$. For classical families of graphs, Kelenc et al. \cite{kele} showed the following results.
\begin{theorem}\emph{\cite{kele}}
The edge metric dimension of certain families are computed as follows:
\begin{itemize}
\item[\emph{(i)}] For a graph $\G$, $\B_e(\G)=1$ if and only if $\G=P_n$.\label{tab:EP}
\item[\emph{(ii)}] For $n\geq3$, we have $\B_e(C_n) = 2$.\label{tab:EC}
\item[\emph{(iii)}] For $n\geq2$, we have $\B_e(K_n) = n-1$.\label{tab:EK}
\item[\emph{(iv)}] For $m,n\geq2$, we have $\B_e(K_{n,m})=n+m-2$.\label{tab:EKB}
\item[\emph{(v)}] For $r\geq t\geq2$, we have $\B_e(\G=P_r \Box P_t) =2 $.\label{thm:edimLn}
\item[\emph{(vi)}] For $n\geq3$, we have
\begin{equation*}
\B_e(W_{1,n})=\left\{
                       \begin{array}{ll}
                         n, & \hbox{$n=3,4$;} \\
                         n-1, & \hbox{$n\geq5$.}
                       \end{array}
                     \right.
\end{equation*}\label{tab:EW}
\item[\emph{(vii)}] For $n\geq1$, we have
\begin{equation*}
\B_e(F_{1,n})=\left\{
                       \begin{array}{ll}
                         n, & \hbox{$n=1,2,3$;} \\
                         n-1, & \hbox{$n\geq4$.}
                       \end{array}
                     \right.
\end{equation*}\label{tab:EF}
\end{itemize}
\end{theorem}
\begin{lemma}\emph{\cite{kele}}\label{thm:lemma}
Let $\G$ be a graph with $\Delta(\G)=n-1$. If there exist at least two vertices with degree $n-1$, then $\B_e(\G) = n-1$.
\end{lemma}
\begin{proposition}\emph{\cite{kele}}\label{thm:edimn-1}
Let $\G$ be a graph of order $n$. If there is a vertex $v \in V(\G)$ of degree $n-1$, then,
either $\B_e(\G)=n-1$ or $\B_e(\G)=n-2$.
\end{proposition}

Filipovic et al. \cite{fili} showed the following result.
\begin{corollary}\emph{\cite{fili}}\label{tab:fili}
Let $\G$ be a connected $r$-regular graph. Then we have $\B_e(\G)\geq 1 + \lceil\log_2 r\rceil.$
\end{corollary}
Peterin and Yero, \cite{Peterin} computed the edge metric dimension of the Corona product and Join of two graphs.
\begin{theorem}
\emph{\cite{Peterin}} Let $\Gamma$ and $\Omega$ be two graphs where $\Gamma$ is connected and $| V(\Omega) | \geq 2$. Then we have
\begin{center}
$\beta_e (\Gamma \odot \Omega) = |V (\Gamma)| \cdot (|V (\Omega)| - 1).$ \label{thm:corona}
\end{center}
\end{theorem}
\begin{theorem}
\emph{\cite{Peterin}} Let $\Gamma$ and $\Omega$ be two connected graphs. Then we have
$\beta_e (\Gamma \vee \Omega) = |V (\Gamma)| + |V (\Omega)| - 1$. \label{thm:join}
\end{theorem}

Zubrilina \cite{zub} showed the following results regarding the edge metric dimension of graphs.
\begin{theorem}\emph{\cite{zub}}\label{thm:theorem}
Let $\G$ be a graph with order $n$. Then $\B_e(\G)=n-1$ if and only if for any distinct $v_1,v_2 \in V(\G)$, there exists $u \in V(\Gamma)$
such that $v_1u, v_2u \in E(\G)$ and $u$ is adjacent to all non-mutual neighbors of $v_1$ and $v_2$.
\end{theorem}
The following corollary is followed from Theorem \ref{thm:theorem}.
\begin{corollary}\emph{\cite{zub}}\label{thm:collary}
Let $\G$ be a graph of order $n$. Suppose $\B_e(\G)=n-1$. Then, $\mathrm{diam}(\G)=2$ and every edge lies on a triangle.
\end{corollary}

Let $T_n =(V,E)$ be a tree and let $v\in V$. Define the equivalence relation $R_v$ in the following way: for every two edges $e,~f$ we let $eR_vf$ if and only if
there is a path in $T_n$ including $e$ and $f$ that does not have $v$ as an internal vertex. The subgraphs induced by the edges of the
equivalence classes of $E$ are called the bridges of $T_n$ relative to $v$. Furthermore, for each vertex $v\in V$, the legs at $v$ are the
bridges which are paths. We denote by $l_v$ the number of legs at $v$.

For a tree $T_n$, Kelenc et al. \cite{kele} showed the following.
\begin{remark}\emph{\cite{kele}}\label{keleremark}
Let $T_n=(V,E)$ be a tree. If $T_n$ is not a path, then
$\B(T_n) = \B_e(T_n) = \sum_{v \in V, l_v>1} (l_v-1).$
\end{remark}

Boutrig et al. \cite{bou} introduced the concept of vertex-edge domination in graphs.
A vertex $x\in V$ in $\G$ is said to $ve-$dominate an edge if it dominates the
edges incident to it as well as the edges adjacent to these incident edges.
In other words, a vertex $w\in V$ in a graph $ve-$dominates an edge $xy \in E$ in $\G$ if
\begin{itemize}
\item[(i)] $w = x$ or $w = y$ ($w$ is incident to $xy$), or
\item[(ii)] $wx$ or $wy$ is an edge in $\G$ ($w$ is incident to an edge adjacent to $xy$).
\end{itemize}
A set $T \subseteq V(\G)$ is a $ve-$dominating set if for all edges $e \in E(\G)$, there exists a vertex $x \in T$ such that
$x$ dominates $e$. The minimum cardinality of a $ve-$dominating set in $\G$ is called the $ve-$domination number and is
denoted by $\g_{ve}(\G)$. A $ve-$dominating set $T$ of cardinality $\g_{ve}(\G)$ is known as a $\g_{ve}$-set.
Peters \cite{Pet} in his PhD thesis showed the following results.
\begin{proposition}\emph{\cite{Pet}}\label{tab:DP}
The vertex-edge domination number of some graph families is computed as follows:
\begin{itemize}
\item[\emph{(i)}] For $m,n\geq2$, we have $\g_{ve}(Kn)=\g_{ve}(K_{m,n})=1$.
\item[\emph{(ii)}] For $n\geq2$, we have $\g_{ve}(P_n) = \lfloor \frac{n+2}{4} \rfloor$.
\item[\emph{(iii)}] For $n\geq3$, we have $\g_{ve}(C_n) =\lfloor \frac{n+3}{4} \rfloor$.
\end{itemize}
\end{proposition}
\begin{proposition}\emph{\cite{Pet}}\label{thm:propo1}
Let $\G$ be a graph of order $n$. Then, $\g_{ve}(\G)=1$ if and only if there exists a vertex
$x\in V(\G)$ such that every vertex of $\G$ is within distance two of $x$ and if
$Y=\{y \in V(\G):d(x,y)= 2\}$ then $Y$ is an independent set of vertices.
\end{proposition}
\begin{proposition}\emph{\cite{Pet}}\label{thm:propo2}
For any graph $\G$ of order $n$, size $m$ and maximum degree $\Delta(\G)$,
$$\Big\lceil \frac{m}{(\Delta(\G))^2}  \Big\rceil \leq \g_{ve}(\G).$$
\end{proposition}
\begin{proposition}\emph{\cite{Pet}}\label{thm:propo3}
For any graph $\G$ of size $m$, maximum degree $\Delta(\G)$, and minimum degree $\delta(\G)$,
$$\g_{ve}(\G) \leq m- \Delta(\G)-\frac{\Delta(\G)(\delta(\G)-1)}{2}+1.$$
\end{proposition}

Lewis \cite{Lew} in his PhD thesis showed the following relation.
\begin{proposition}\emph{\cite{Lew}}\label{prop0}
For any graph $\G$, we have $\g_{ve}(\G) \leq \g(\G)$.
\end{proposition}

\section{The vertex-edge dominant edge metric dimension of graphs}\label{sec3}
By combining the ideas of a vertex-edge dominating set and an edge resolving set,
we introduce a vertex-edge dominant edge resolving set. A subset $W \subset V(\G)$
is said to be a $ve$-dominant edge resolving set, if $W$ is a $ve-$dominating set
and an edge resolving set simultaneously. Minimal cardinality of such a set in $\G$
is called the vertex-edge dominant edge metric dimension of $\G$. We denote this
parameter by $\g_{emd}(\G)$.

In order to explain the idea, we present an example. We consider $\G$ as the graph with vertex set
$V(\G)= \{v_1, v_2, \dots, v_8\}$ and edge set $E(\G)= \{a, b, c, d, e, f, g, h\}$ as shown in Figure \ref{fig000}.
Note that the set $W = \{v_1, v_4, v_5 \}$ is a vertex-edge dominant edge resolving set of $\G$.

\begin{figure}[htbp!]
\centering
  \includegraphics[width=5cm]{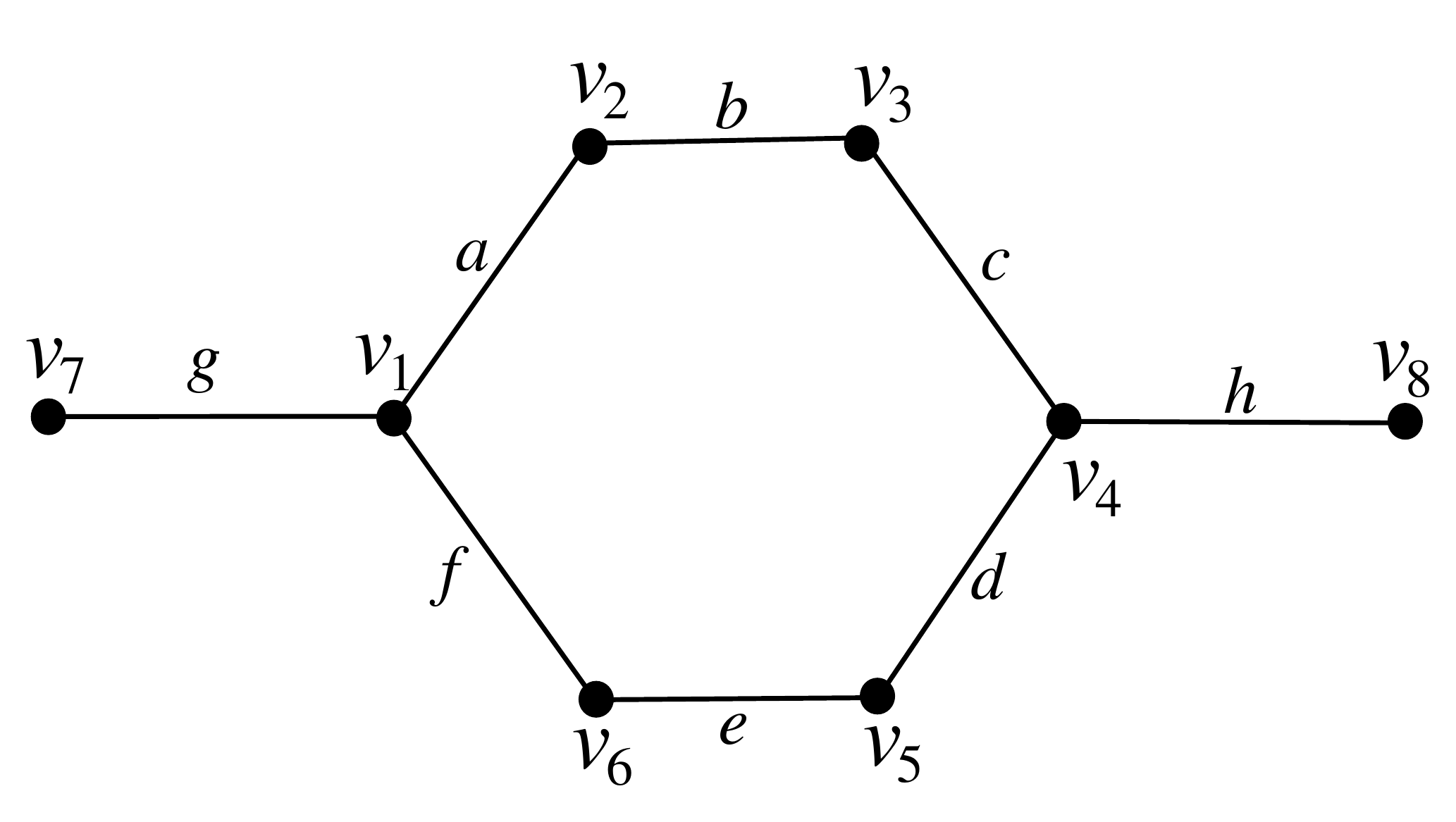}
  \caption{A graph with the $\g_{emd}(G)= 3$.}\label{fig000}
\end{figure}

The vertex $v_1$ will dominate the edges $a, b, g, f$ and $e$ and the vertex $v_4$ will dominate the
edges $b, c, d, e$ and $h$. However, the vertices $v_1$ and $v_4$ do not resolve all the edges $\G$ since, for example, $r(a | \{v_1, v_4 \}) = r(f | \{v_1, v_4 \})$, $r(b | \{v_1, v_4 \}) = r(e | \{v_1, v_4 \})$ and $r(c | \{v_1, v_4 \}) = r(d | \{v_1, v_4 \})$.
Table \ref{tab:MRS} exhibits the representations of all the edges of $\G$ corresponding to $W$.
Routine calculations show that no set of cardinality smaller than three is a vertex
edge dominant edge resolving set of $\G$. This leads us to the fact that $\g_{emd}(G)= 3$.

\begin{table}[htbp!]
		\centering
		\begin{tabular}{|c|c|c|c|c|c|c|c|c|c|c|c|c|}
		\hline
	    Edges & $a$ & $b$& $c$ &$d$ & $e$ & $f$ & $g$ & $h$\\ \hline
		${\small \g_{emd}r(..|T)}$ & $\{0,2,2\}$& $\{1,1,2\}$& $\{2,0,1\}$& $\{2,0,0\}$& $\{1,1,0\}$& $\{0,2,1\}$& $\{0,3,1\}$& $\{3,0,1\}$\\ \hline
		
        \end{tabular}
        \vspace*{0.3cm}
        \caption{The representations of the edges of $\G$ with respect to the set $W=\{v_1, v_4, v_5\}$.}\label{tab:MRS}
	\end{table}

Next, we compute the vertex-edge dominant edge metric dimension of classical families of graphs.
\subsection{Path and cycle graphs}

The following theorems compute the vertex-edge dominant edge metric dimension of path graph $P_n$ and cycle graph $C_n$.
\begin{theorem}
Let $P_n$ be a path graph, where $n\geq2$. Then
$$\g_{emd}(P_n) =
\begin{cases}
\B_{e}(P_n) = 1, & n = 2, 3;\\
\B_{e}(P_n) + 1 = 2, & n = 4, 5;\\
\g_{ve}(P_n) = \lfloor \frac{n+2}{4} \rfloor, & n \geq 6.
\end{cases}$$
\end{theorem}
\proof
Let $V(P_n)=\{v_1, v_2, \ldots, v_n\}$ with $n-1$ edges and $E(P_n=\{e_i = v_iv_{i+1} ; 1 \leq i \leq n-1 \})$. Recall that, by Theorem \ref{tab:EP}(i) (resp. Proposition \ref{tab:DP}) we have $\B_e(P_n) = 1$ (resp. $\g_{ve}(P_n) = \lfloor \frac{n+2}{4} \rfloor$). We divide the proof into a number of cases:

\noindent{\textbf{Case 1:} If $n=2$, then $W = \{v_1\}$ will dominate and gave a unique code to the edge $e_1$.

\noindent{\textbf{Case 2:} If $n=3$, then $W = \{v_1\}$ will dominate and gave a unique code to the edges $e_1$ and $e_2$.

\noindent{\textbf{Case 3:} If $n=4$, then there are four vertices and three edges. Without loss of generality, we assume that $v_1, v_2, v_3, v_4$
be the four vertices and $e_1, e_2, e_3$ be the three edges. In this case, the vertex $v_2$
will dominate all the edges, however, does not give a unique code to the all edges. As, $r(e_1 | v_2) = r(e_2 | v_2)$,
the edge $e_1$ and the edge $e_2$ have same code or representation. Similarly, the vertex $v_3$
will dominate all the edges, however, does not give a unique code to the all edges. As, $r(e_2 | v_3) = r(e_3 | v_3)$,
the edge $e_2$ and the edge $e_3$ have same code or representation. So, the set $W = \{v_1,v_2\}$ will dominate and
provide a unique code to all the edges of $P_4$ and we obtain that $\g_{emd}(P_4) = \B_{e}(P_4) +1 = 2$.

\noindent{\textbf{Case 4:} For $n=5$, there are five vertices, say, $v_1, v_2, v_3, v_4, v_5$ and four edge, say, $e_1, e_2, e_3, e_4$.
In this case, the vertex $v_3$ will dominate all edges, however, does not yield a unique code to the all edges. The edge $e_2$ and the edge $e_3$
would have the same representation. So, the set $W = \{v_1,v_3\}$ will dominate and gives a unique code to the all edges of $P_5$
and we obtain that $\g_{emd}(P_5) = \B_{e}(P_5) +1=2$.

\noindent{\textbf{Case 5:} Assume $n\geq6$. Let $W\subset V(P_n)$, such that
$$W =
\begin{cases}
\{v_{2+4i} : 0 \leq i < \lfloor \frac{n+2}{4} \rfloor \}, & n \equiv 0(\mod 4);\\
\{v_{3+4i} : 0 \leq i < \lfloor \frac{n+2}{4} \rfloor \}, & n \equiv 1(\mod 4);\\
\{v_{1+4i} : 0 \leq i < \lfloor \frac{n+2}{4} \rfloor \}, & n \equiv 2,3(\mod 4).
        \end{cases}$$

Since $|W|= \lfloor \frac{n+2}{4} \rfloor$, by Proposition \ref{tab:DP}, the set $W$ dominates all the edges of $P_n$ and no two edges have the same representation with respect
to $W$. We obtain the desired result.
\wbull

\begin{theorem}
Let $C_n$ ba a cycle graph with $n \geq 3$. Then
$$\g_{emd}(C_n) =
\begin{cases}
\B_e(C_n) = 2, & 3\leq n \leq7;\\
\B_e(C_n)+ 1 = 3, & n=8;\\
\g_{ve}(C_n) = \lfloor\frac{n+3}{4}\rfloor, &n \geq 9.
        \end{cases}$$
\end{theorem}
\proof
Let $V(C_n)=\{u_1,u_2,\ldots,u_n \}$ and $E(C_n)=\{e_i= u_iu_{i+1}; 1 \leq i \leq n, u_{n+1}=u_1\}$.
Recall that, by Theorem \ref{tab:EP}(ii) (resp. Proposition \ref{tab:DP}) we have $\B_e(C_n) = 2$
(resp. $\g_{ve}(C_n) = \lfloor\frac{n+3}{4}\rfloor$). We divide the proof into a number of cases:

\noindent{\textbf{Case 1:} If $n=3$, then the vertex $u_1$, will dominate all edges of $C_3$, however, the edges
$e_1$ and $e_2$ would have the same representation as $\B_e(C_3) = 2$. The set $W =\{u_1, u_2\}$ is a vertex-edge dominant
edge metric basis. So, we get $\g_{emd}(C_3) = 2$.

\noindent{\textbf{Case 2:} If $n=4$, then the vertex $u_1$ will dominate all edges of $C_4$ but the edges
$e_1$ (resp. $e_2$) would have the same representation with $e_4$ (resp. $e_3$) because of the fact that $\B_e(C_4) = 2$.
The set $W =\{u_1, u_2\}$ is a vertex-edge dominant edge metric basis. Therefore, we obtain $\g_{emd}(C_4) = 2$.

\noindent{\textbf{Case 3:} If $n=5$, then any vertex $u \in V(C_5)$ will not dominate all edges of $C_5$ as $\g_{ve}(C_5) = 2$ by Proposition \ref{tab:DP}.
However, the set $W =\{u_1, u_4\}$ is a vertex-edge dominant edge metric basis as it would generate distinct codes
for all the edges of $C_5$. This implies that $\g_{emd}(C_5) = 2$.

\noindent{\textbf{Case 4:} If $n=6$ or $n=7$, then similar to Case 3, a single vertex $u \in V(C_6)$ or $u \in V(C_7)$ can never dominate
all edges of $C_6$ or $C_7$ as $\g_{ve}(C_6)=\g_{ve}(C_7) = 2$ by Proposition \ref{tab:DP}. However, the set $W =\{u_1, u_5\}$ is a vertex-edge dominant edge metric
basis as it would generate distinct codes for all the edges of $C_6$ and $C_7$. This implies that $\g_{emd}(C_6)=\g_{emd}(C_7) = 2$.

\noindent{\textbf{Case 5:} If $n=8$, then the vertices $u_1$ and $u_5$ dominate all edges of $C_8$. However, they do not give
different codes of representation to all the edges as $r(e_j | \{u_1, u_5 \}) = r(e_{n+j-1} | \{u_1, u_5 \})$. Note that $W=\{u_1, u_3, u_5\}$ is a vertex-edge dominant edge metric
basis as it would generate distinct codes for all the edges of $C_8$. This implies that $\g_{emd}(C_8) = 3$.

\noindent{\textbf{Case 6:} Assume $n\geq9$. Let $W\subset V(C_n)$, such that
$$W =
\begin{cases}
\{u_{1+5i} : 0 \leq i < \lfloor \frac{n+3}{4} \rfloor \}, & n \equiv 0, 2, 3(\mod 4);\\
\{u_{n-1}, u_{1+5i} : 0 \leq i < \lfloor \frac{n+3}{4} \rfloor -1 \}, & n \equiv 1(\mod 4).
        \end{cases}$$

Since, $|W|= \lfloor \frac{n+3}{4} \rfloor$, by Proposition \ref{tab:DP}, the set $W$ dominates all the edges of $C_n$ and no two edges have the same representation with respect
to $W$. We obtain the desired result.
\wbull

\subsection{Complete graphs and complete bipartite graphs}
The following theorem computes the vertex-edge dominant edge metric dimension of the complete graph $K_n$ and the complete bipartite graph $K_{n,m}$.
\begin{theorem}
Let $K_n$ be the complete graph with order $n \geq 2$. Then $\g_{emd}(K_n) = \B_{e}(K_n) = n-1$.
\end{theorem}
\proof
By Proposition \ref{tab:DP}, we have $\g_{ve}(K_n)=1$. Any vertex $u \in V(K_n)$
will dominate all edges of the graph $K_n$ but, by Theorem \ref{tab:EK}(iii), the edge dimension of $K_n$ is $n-1$ i.e.
$\B_e(K_n) = n-1$ . Thus, any vertex from an edge metric basis for the graph $K_n$ will dominate all
edges of the complete graph $K_n$. Therefore, we obtain $\g_{emd}(K_n)=\B_{e}(K_n)=n-1$.
\wbull

\begin{theorem}\label{g-emd-kmn}
Let $K_{n,m}$ be the complete bipartite graph, where $n,m\geq2$. Then $\g_{emd}(K_{n,m})=\B_{e}(K_{n,m})=n+m-2$.
\end{theorem}
\proof
By Proposition \ref{tab:DP}, we have $\g_{ve}(K_{n,m})=1$. Any vertex $u \in V(K_{n,m})$ will
dominate all edges of the graph $K_{n,m}$ but, by Theorem \ref{tab:EKB}(iv), the edge dimension of complete bipartite graph $K_{n,m}$ is $n+m-2$
i.e. $\B_e(K_{n,m}) = n+m-2$. Thus, any vertex from an edge metric basis for the graph
$K_{n,m}$ will dominate all edges of the complete bipartite graph $K_{n,m}$. Therefore, we obtain $\g_{emd}(K_{n,m})=\B_{e}(K_{n,m})=n+m-2$.
\wbull

As a special case to Theorem \ref{g-emd-kmn}, we obtain the following.
\begin{corollary}
Let $S_{1,n}$ be the star graph, where $n \geq 3$. Then $\g_{emd}(S_{1,n})=\B_{e}(S_{1,n})=n-1$.
\end{corollary}

\subsection{Wheel and fan graphs}

For the ($n+1$)-dimensional wheel and fan graphs, first we compute its dominant metric dimension and the vertex-edge domination number.

\begin{proposition}
Let $\G$ be a wheel graph $W_{1,n}$ or fan graph $F_{1,n}$. Then,
\begin{equation*}
\g_{md}(\G)=\left\{
                       \begin{array}{ll}
                         \B(\G), & \hbox{$n \equiv 0,2, 4 (\mod 5)$;} \\
                         \B(\G) + 1, & \hbox{$n \equiv 1, 3 (\mod 5)$.}
                       \end{array}
                     \right.
\end{equation*}
\end{proposition}
\proof
Let $V(\G) = \{x,y_1,y_2,\ldots,y_n\}$, where the vertex $x$ has degree $n$ and the vertices $y_1,y_2,\ldots,y_n$ induce a cycle $C_n$ in wheel graph $\G = W_{1,n}$ (
or induce a path $P_n$ in fan graph $\G = F_{1,n}$). Let $S = \{x\}$. Then, every vertex of $\G$ is dominated by $S$. This implies that $\g(\G) = 1$.  Buczkowski et al. \cite{buc} showed that $\B(\G = W_{1,n}) = \lfloor \frac{2n+2}{5} \rfloor$, where $n\geq 7$ and Hernando et al. \cite{her} showed that $\B(\G =F_{1,n}) = \lfloor \frac{2n+2}{5} \rfloor$, where $n\geq7$. Based on resolvability of vertices in $\G$, we further divide the proof
into five cases:

\noindent{\textbf{Case 1:} Assume that $n\equiv0(\mod 5)$.

Let $n = 5k$ where $k \geq 2$ and $\lfloor \frac{2n+2}{5} \rfloor = 2k$.
Note that for $D =\{y_{5i+2},y_{5i+4}:0\leq i \leq k-1\}$, $V(\G)\backslash D = N(D)$. Since $| D | = \lfloor \frac{2n+2}{5} \rfloor$.
This implies that $D$ is a dominant resolving set for $\G$ where $n\equiv 0 (\mod 5)$.
Thus, $\g_{md}(\G) = \B(\G)$ for $n\equiv 0 (\mod 5)$.

\noindent{\textbf{Case 2:} Assume $n\equiv1(\mod 5)$.

Let $n = 5k+1$ where $k \geq 2$ and $\lfloor \frac{2n+2}{5} \rfloor = 2k$.
Note that for $D_1 = \{ y_{5i+2}, y_{5i+4} : 0 \leq i \leq k-1 \}$, we have
$V(\G)\backslash D_1 \neq N(D_1)$, because the vertex $y_{5k+1}$ is not dominated by any vertex of the set $D_1$. The set $D_1$ is resolving set but is not dominating set for $\G$. This implies that $D_1$ is not a dominant resolving set for $\G$, where $n\equiv 1(\mod 5)$.
Note that for $D_2 = \{ y_{5i+2}, y_{5i+4} : 0 \leq i \leq k-2 \} \cup \{y_{5k-3}, y_{5k+1}\}$, we have
$V(\G)\backslash D_2 \neq N(D_2)$, because the vertex $y_{5k-1}$ is not dominated by any vertex of the set $D_2$. The set $D_2$ is resolving set but is not dominating set for $\G$. This implies that $D_2$ is not a dominant resolving set for $\G$, where $n\equiv 1(\mod 5)$.
Note that for $D_3 = \{ y_{5i+2}, y_{5i+4} : 0 \leq i \leq k-2 \} \cup \{y_{5k-3}, y_{5k}\}$, we have
$V(\G)\backslash D_3 = N(D_3)$. The set $D_3$ is dominating set but is not resolving set for $\G$ as $r (y_{5k-1} | D_3) = r(y_{5k+1} | D_3) $. This implies that $D_3$ is not a dominant resolving set of $\G$, where $n\equiv 1(\mod 5)$.
Note that for $D = S \cup D_1 = \{x\} \cup \{y_{5i+2}, y_{5i+4} : 0 \leq i \leq k-1 \}$, we have $V(\G)\backslash D = N(D)$, the set $D$ is dominating and resolving set for $\G$. This implies that $D$ is a dominant resolving set for $\G$, where $n\equiv 1(\mod 5)$. Thus, we obtain $\g_{md}(\G) = \B(\G)+1$, where $n\equiv 1(\mod 5)$.

\noindent{\textbf{Case 3:} Assume $n\equiv2(\mod 5)$.

Let $n = 5k+2$ where $k \geq 1$ and $\lfloor \frac{2n+2}{5} \rfloor = 2k+1$.
For $D = \{ y_{5i+2}, y_{5i+4} : 0 \leq i \leq k-1 \} \cup \{y_{5k+2}\}$ we have $V(\G)\backslash D = N(D)$. This
implies that $D$ is a dominant resolving set of $\G$ where $n\equiv 2(\mod 5)$.
Therefore, we obtain $\g_{md}(\G) = \B(\G)$, where $n\equiv 2(\mod 5)$.

\noindent{\textbf{Case 4:} Assume $n\equiv3(\mod 5)$.

Let $n = 5k+3$ where $k \geq 1$ and $\lfloor \frac{2n+2}{5} \rfloor = 2k+1$.
For $D_1 = \{ y_{5i+2}, y_{5i+4} : 0 \leq i \leq k-1 \} \cup \{y_{5k+3}\}$ we have
$V(\G)\backslash D_1 \neq N(D_1)$, because the vertex $y_{5k+1}$ is not dominated by any vertex of the set $D_1$. The set $D_1$ is resolving set but is not dominating set for $\G$.
This implies that the set $D_1$ is not a dominant resolving set for $\G$ where $n\equiv 3 (\mod 5)$.
For $D_2 = \{ y_{5i+2}, y_{5i+4} : 0 \leq i \leq k-1 \} \cup \{y_{5k+1}\}$ we have
$V(\G)\backslash D_2 \neq N(D_2)$, because the vertex $y_{5k+3}$ is not dominated by any vertex of the set $D_2$. The set $D_2$ is resolving set but is not dominating set for $\G$.
This implies that the set $D_2$ is not a dominant resolving set for $\G$ where $n\equiv 3 (\mod 5)$.
For $D_3 = \{ y_{5i+2}, y_{5i+4} : 0 \leq i \leq k-1 \} \cup \{y_{5k+2}\}$ we have
$V(\G)\backslash D_3 = N(D_3)$. The set $D_3$ is dominating set but is not resolving set for $\G$ as $r(y_{5k+1} | D_3) =r(y_{5k+3} | D_3) $. This implies that $D_3$ is not a dominant resolving set for $\G$, where $n\equiv 1(\mod 5)$.
Note that for $D = S \cup D_1 = \{x\} \cup \{ y_{5i+2}, y_{5i+4} : 0 \leq i \leq k-1 \} \cup \{y_{5k+3}\} $, we have $V(\G)\backslash D = N(D)$. This implies that the set $D$ is a dominant resolving set for $\G$ where $n\equiv 3 (\mod 5)$. Thus, we obtain $\g_{md}(\G) = \B(\G) + 1$, where $n\equiv3(\mod 5)$.

\noindent{\textbf{Case 5:} Assume $n\equiv4(\mod 5)$.

Let $n = 5+4k$ where $k \geq 1$ and $\lfloor \frac{2n+2}{5} \rfloor = 2k+2$.
Note that if $D = \{ y_{5i+2}, y_{5i+4} : 0 \leq i \leq k\}$, then we have $V(\G)\backslash D = N(D)$.
This implies that $D$ is a dominant resolving set for $\G$ where $n\equiv4(\mod 5)$.
Therefore, $\g_{md}(\G) = \B(\G)$, where $n\equiv4(\mod 5)$. This completed the proof.
\wbull

The following proposition computes the vertex-edge domination number of wheel graphs.
\begin{proposition}
For the wheel graph $W_{1,n}$, fan graph $F_{1,n}$ and star graph $S_{1,n}$, we have $$\g_{ve}(W_{1,n}) = \g_{ve}(F_{1,n}) =\g_{ve}(S_{1,n}) = 1.$$
\end{proposition}
\proof
Assume that  $V(W_{1,n}) = V(F_{1,n}) =V(S_{1,n}) =\{x,y_1,y_2,\ldots,y_n\}$,
where the vertex $x$ has degree $n$ and the vertices $y_1,y_2,\ldots,y_n$ induce a cycle $C_n$ in wheel graph and
induce a path $P_n$ in the fan graph. In $S_{1,n}$, all the edges are incident to the vertex $x$ so, the set $T= \{x\}$
dominate all edges of star graph $S_{1,n}$. In the wheel and the fan graph, all edges are incident to $N(x)=\{y_1,y_2,\ldots,y_n\}$
so, the set $T= \{x\}$ dominate all edges of $W_{1,n}$ and $F_{1,n}$. Therefore, the set $T= \{x\}$ dominate
all edges of the wheel graph $W_{1,n}$, the fan graph $F_{1,n}$ and the star graph $S_{1,n}$. Thus, we obtain
$\g_{ve}(W_{1,n}) = \g_{ve}(F_{1,n}) =\g_{ve}(S_{1,n}) = 1$.
\wbull

The next theorem calculates the vertex-edge dominant edge metric dimension of wheel graphs.
\begin{theorem}
Let $W_{1,n}$ be a wheel graph $W_{1,n}$ and $F_{1,n}$ be a fan graph with $n \geq 5$. Then
$$\g_{emd}(W_{1,n}) = \g_{emd}(F_{1,n}) = n-1.$$
\end{theorem}
\proof
For $n \geq 5$, the edge metric dimension of wheel graph is $\B_e(W_{1,n}) = n-1$ by Theorem \ref{tab:EW}(vi) and the edge metric dimension of fan graph is $\B_e(F_{1,n}) = n-1$ by Theorem \ref{tab:EF}(vii). The set $W = \{y_1,y_2,\ldots, y_{n-1}\}$ is an edge metric basis/generator. The set $W$ dominates all the
edges of wheel graph $W_{1,n}$ and fan graph $F_{1,n}$. The set $W$ is also the vertex-edge dominant set. Therefore, we obtain that
$\g_{emd}(W_{1,n}) = \g_{emd}(F_{1,n}) =\B_{e}(W_{1,n}) = \B_{e}(W_{1,n}) = n-1$, if $n\geq5$.
\wbull

\section{The vertex-edge dominant edge metric dimension of Cartesian products}\label{sec4}

In this section, we study the vertex-edge dominant edge metric dimension in some standard Cartesian products such as $P_n \Box P_r$ and $C_n \Box P_r$.

The graph $\G =P_n \Box P_2$ is the Cartesian product of the graphs $P_n$ and $P_2$.
The vertex set and edge set of $\G$ are $V(P_n \Box P_2)=\{u_i, vj; 1 \leq i,j \leq n\}$ and
$E(P_n \Box P_2) = \{ a_i=u_iu_{i+1}, b_j=v_jv_{j+1}; 1 \leq i,j \leq n-1 \} \cup \{c_i=u_iv_j; i=j, 1 \leq i,j \leq n\}$,
respectively.

\begin{proposition}\label{propo:vedominationLn}
Let $\G =P_n\Box P_2$ be the Cartesian product of $P_n$ and $P_2$. Then, we have
$$\g_{ve}(\G) = \Big\lceil \frac{n}{3} \Big\rceil.$$
\end{proposition}
\proof
Note that the order (resp. size) of $\G$ is $2n$ (resp. $3n-2$). Furthermore, degrees of all vertices of $\G$ are 3
except the vertices $u_1, u_n, v_1, v_n$ which are of degree 2. We divide the proof into a number of cases:

\noindent{\textbf{Case 1:} Assume that $n=1,2,3$.

For $n=1,2$, the singleton set $T =\{u_1\}$ is a minimum vertex-edge dominant set, whereas for $n=3$,
the singleton set $T =\{u_2\}$ is a minimum vertex-edge dominant set.

\noindent{\textbf{Case 2:} Assume that $n=4,5,6$.

The set $T =\{u_1, v_4\}$ (resp. $T =\{u_2, v_5\}$) is a minimum vertex-edge dominant set
if $n=4,5$ (resp. $n=6$).

\noindent{\textbf{Case 3:} Assume that $n\geq7$ and $n\equiv0(\mod6)$.
The minimum vertex-edge dominant set in this case is
$$T = \{u_{2+6i}; 0 \leq i < \Big\lceil \frac{n}{6} \Big\rceil \} \cup \{v_{5+6j}; 0 \leq j < \Big\lceil \frac{n}{6} \Big\rceil\}.$$

\noindent{\textbf{Case 4:} Assume that $n\equiv1,2(\mod6)$ and $n\geq7$.
The minimum vertex-edge dominant set in this case is
$$T = \{u_{1+6i}; 0 \leq i < \Big\lceil \frac{n}{6} \Big\rceil\} \cup \{v_{4+6j}; 0 \leq j < \Big\lceil \frac{n}{6} \Big\rceil -1\}.$$

\noindent{\textbf{Case 5:} Assume that $n\equiv3(\mod6)$ and $n\geq7$.
The minimum vertex-edge dominant set in this case is
$$T = \{u_{2+6i}; 0 \leq i < \Big\lceil \frac{n}{6} \Big\rceil\} \cup \{v_{5+6j}; 0 \leq j < \Big\lceil \frac{n}{6} \Big\rceil -1\}.$$

\noindent{\textbf{Case 6:} Assume that $n\equiv4,5(\mod6)$ and $n\geq7$. The set
$$T = \{u_{1+6i}; 0 \leq i < \Big\lceil \frac{n}{6} \Big\rceil\} \cup \{v_{4+6j}; 0 \leq j < \Big\lceil \frac{n}{6} \Big\rceil \}.$$
provides the minimum vertex-edge dominant set of $\G$.
From Case 3, Case 4, Case 5 and Case 6 it is clear that the $\g_{ve}(\G) \leq \lceil \frac{n}{3} \rceil$. Since, the maximum degree of $\G$ is $3$ i.e. $\Delta(\G) = 3$, by Proposition \ref{thm:propo2}, we obtained that $\g_{ve}(\G) \geq \lceil \frac{n}{3} \rceil$. This show that $\g_{ve}(\G) = \lceil \frac{n}{3} \rceil$.
\wbull

\begin{theorem}
Let $\G =P_n\Box P_2$ be the Cartesian product of $P_n$ and $P_2$. Then, we have
$$\g_{emd}(\G) =
\begin{cases}
\B_{e}(\G )=1 & n = 1;\\
\B_{e}(\G)=2 & n = 2 ;\\
\B_{e}(\G)+1=3 & n = 3,4,5,6; \\
\g_{ve}(\G) = \lceil\frac{n}{3}\rceil &  n\geq7.
\end{cases}$$
\end{theorem}
\proof
By Theorem \ref{thm:edimLn} (resp. Proposition \ref{propo:vedominationLn}), we have that $\B_e(\G =P_n \Box P_2) = 2$
(resp. $\g_{ve}(\G)=\lceil\frac{n}{3}\rceil$). Next, we divide the proof into a number of cases:

\noindent{\textbf{Case 1:} Assume that $n=1,2,3$.

Note that we have $\B_e(\G)=1$ (resp. $\B_e(\G)=2$) for $n=1$ (resp. $n=2,3$).
Moreover, $\g_{ve}(\G)=1$ for $n=1,2,3$. Then, $W =\{u_1\}$ provide unique
representations to every edge of $\G$ if $n=1$. Similarly, $T=\{u_1, u_2\}$
(resp. $T=\{u_1, u_2, u_3\}$)) provide distinct codes to the edges of $\G$
if $n=2$ (resp. $n=3$). Thus, we obtain
$$\g_{emd}(\G) =
\begin{cases}
1, & n = 1;\\
2, & n = 2 ;\\
3, & n = 3.
\end{cases}$$

\noindent{\textbf{Case 2:} Assume that $n=4,5,6$.

Note that, for $n=4$, the set $A=\{u_1, u_4 \}$ is a minimum edge resolving set of $\G$.
In a similar manner, the set $A=\{u_1,u_5\}$ (resp. $A=\{u_1,u_6\}$) provides a minimum
edge resolving set for $n=5$ (resp. $n=6$). Therefore, we obtain that $\B_e(\G)=2$, if $n=4,5,6$.
Moreover, the set $T =\{u_1,v_4\}$ is a minimum vertex-edge dominating set for $P_4 \Box P_2$.
And, the set $T =\{u_1,v_5\}$ (resp. $T =\{u_2,v_5\}$) is a minimum vertex-edge dominating set
for $P_5\Box P_2$ (resp. $P_6\Box P_2$). This implies that $\g_{ve}(\G)=2$ if $n=4,5,6$.
However, the set $A$ can not dominate the edge $b_2=v_2v_3$ if $n=4$, the edges $b_2=v_2v_3, b_3=v_3v_4$
if $n=5$ and the edges $a_3=u_3u_4, b_2=v_2v_3, b_3=v_3v_4, b_4=v_4v_5, c_3=u_3v_3, c_4=u_4v_4$ if $n=6$.
Consequently, the set $T$ does not resolve all the edges of $\G$ for $n=4,5,6$. Additionally, the set
\begin{equation*}
W=\left\{
  \begin{array}{ll}
    \{u_1, u_4, v_4\}, & \hbox{$n = 4$;} \\
    \{u_1, u_5, v_5\}, & \hbox{$n = 5$;} \\
    \{u_2, u_6, v_5\}, & \hbox{$n = 6$.}
  \end{array}
\right.
\end{equation*}
dominates and delivers unique codes to all the edges of $P_n \Box P_2$, if $n=4,5,6$.
This implies that $\g_{emd}(\G)=\B_{e}(\G)+1=3$, for $n=4,5,6$.

\noindent{\textbf{Case 3:} Finally, we assume that $n\geq7$.

The set $A=\{u_1, u_n \}$ is a minimum edge resolving set.
Thus, we have $\B_e(\G)=2$. By Proposition \ref{propo:vedominationLn}, the vertex-edge domination number of $\G$ is
$\g_{ve}(\G)=\lceil\frac{n}{3}\rceil$. Note that the minimum vertex-edge dominating set is also a minimum
edge resolving set for $P_n \Box P_2$. Thus, we obtain that $\g_{emd}(\G)=\g_{ve}(\G)=\lceil \frac{n}{3}\rceil$,
if $n\geq7$.
\wbull

\begin{proposition}\label{thm:edimDn}
Let $\G =C_n \Box P_2$ be the Cartesian product of $C_n$ and $P_2$, with $n\geq3$. Then,
$\B_e(\G)=3$.
\end{proposition}
\proof
Let $V(\G)=\{x_i, y_i ; 1 \leq i \leq n\}$ and $E(\G)= \{e_i=x_ix_{i+1}, f_i=y_iy_{i+1}, g_i=x_iy_i ;1 \leq i \leq n\}$.
Clearly the order (resp. size) of $\G$ is $2n$ (resp. $3n$). Let $A=\{x_1, x_2, x_{l+1} \}$
where $l=\lceil\frac{n}{2}\rceil$ be an edge metric basis. We divide the proof further into some cases:\\

\noindent{\textbf{Case 1:} Assume that $n\equiv0(\mod2)$.
In this case, the representations of edges of $\G$ is as follow:

\noindent{\textit{Subcase 1.1:} Assume $e_i=x_ix_{i+1}$ for $i=1,2,3,\ldots,n$. Then,
$$r(e_i| A) =
\begin{cases}
(0, 0, l-i) & i = 1;\\
(i-1, i-2, l-i) &2 \leq i \leq l;\\
(l-1, l-1, i-l-1) & i=l+1;\\
(n-i, n-i+1, i-l-1) & l+2 \leq i \leq n.
\end{cases}$$
\noindent{\textit{Subcase 1.2:} Assume $f_i=y_iy_{i+1}$ for $i=1,2,3,\ldots,n$. Then,
$$r(f_i| A) =
\begin{cases}
(1, 1, l ) & i = 1;\\
(i, i-1, l-i+1) &2 \leq i \leq l;\\
(l, l, i-l) & i=l+1;\\
(n-i+1, n-i+2, i-l) & l+2 \leq i \leq n.
        \end{cases}$$
\noindent{\textit{Subcase 1.3:} Assume $g_i=x_iy_i$ for $i=1,2,3,\ldots,n$. Then,
$$r(g_i| A) =
\begin{cases}
(0, 1, l ) & i = 1;\\
(i-1, i-2, l-i+1) &2 \leq i \leq l+1;\\
(n-i+1, n-i+2, i-l-1) & l+2 \leq i \leq n.
        \end{cases}$$
For the case where $n$ is even, it is easily observed that no two edges have the same representation. Therefore,
$\B_e(\G)\leq3$, if $n\equiv0(\mod2)$.

\noindent{\textbf{Case 2:} Assume that $n\equiv1(\mod2)$.
The representations for the edges of $\G$ in this cases are as follow:

\noindent{\textit{Subcase 2.1:} Assume $e_i=x_ix_{i+1}$ for $i=1,2,3,\ldots,n$. Then,
$$r(e_i| A) =
\begin{cases}
(0, 0, l-i ) & i = 1;\\
(i-1, i-2, l-i) &2 \leq i \leq l;\\
(n-i, n-i+1, i-l-1) & l+1 \leq i \leq n.
        \end{cases}$$
\noindent{\textit{Subcase 2.2:} Assume $f_i=y_iy_{i+1}$ for $i=1,2,3,\ldots,n$. Then,
$$r(f_i| A) =
\begin{cases}
(1, 1, l ) & i = 1;\\
(i, i-1, l-i+1) &2 \leq i \leq l;\\
(n-i+1, n-i+2, i-l) & l+1 \leq i \leq n.
        \end{cases}$$
\noindent{\textit{Subcase 2.3:} Assume $g_i=x_iy_i$ for $i=1,2,3,\ldots,n$. Then,
$$r(g_i| A) =
\begin{cases}
(0, 1, l-1 ) & i = 1;\\
(i-1, i-2, l-i+1) &2 \leq i \leq l;\\
(l-1,l-1,i-l-1) & i=l+1;\\
(n-i+1, n-i+2, i-l-1) & l+2 \leq i \leq n.
        \end{cases}$$
Thus, for the case where $n$ is odd, it can be seen that no two edges have the same representation. Therefore,
$\B_e(\G)\leq3$, if $n\equiv0,1(\mod2)$.
Since $\G$ is 3-regular graph i.e. $r = 3$, by Corollary \ref{tab:fili}, we obtain that $\B_e(\G) \geq 3$.
This shows that $\B_e(\G)=3$ for any $n$. This completes the proof.
\wbull

\begin{proposition}\label{thm:vedominationDn}
Let $\G =C_n \Box P_2$ be the Cartesian product of $C_n$, $n \geq 4$ and $P_2$. Then,
$\g_{ve}(\G) = \Big\lceil \frac{|E(\G)|}{9} \Big\rceil$.
\end{proposition}
\proof
Note that $\G$ is cubic graph, i.e. every vertex has exactly three edges incident with it.
Therefore, by definition of vertex-edge domination, each vertex of $\G$ will dominate nine
of its edges. For instance, the vertex $x_i$ (resp. $y_i$) will dominate the 9 edges
$e_{i-2}, e_{i-1}, e_{i}, e_{i+1}, f_{i-1}, f_{i}, g_{i-1}, g_{i}, g_{i+1}$
(resp. $f_{i-2}, f_{i-1}, f_{i}, f_{i+1}, e_{i-1}, e_{i}, g_{i-1}, g_{i}, g_{i+1}$).
Moreover, all these edges are incident to the vertices of $N(x_i)$ and $N(y_i)$, respectively.
This implies that the vertex-edge domination number of $\G$ is at least $\lceil\frac{|E(\G)|}{9}\rceil$.
We divide the vertex-edge domination set of $\G$ into a number of cases:

\noindent{\textbf{Case 1:} Let $4\leq n\leq9$.

If $n = 4, 5, 6$ (resp. $n = 7, 8, 9$), then $T =\{x_1, y_4\}$ (resp. $T =\{x_1, y_4, x_7\}$)
is a minimum vertex-edge domination set of $\G$.

\noindent{\textbf{Case 2:} Let $n \geq 10$.
We divide the case further into two subcases:

\noindent{\textit{Subcase 2.1:} Assume that $n \equiv 0, 4, 5(\mod6)$. The minimum vertex-edge dominant set in this case is
$$T = \{x_{1+6i}, y_{4+6i}; 0 \leq i \leq \lceil \frac{n}{6} \rceil -1\}.$$

\noindent{\textit{Subcase 2.2:} Assume that $n \equiv 1, 2, 3(\mod6)$. The set
$$T = \{x_{1+6i}; 0 \leq i \leq \lceil \frac{n}{6} \rceil -1\} \cup \{y_{4+6i}; 0 \leq i \leq \lceil \frac{n}{6} \rceil -2\}.$$
is a minimum vertex-edge domination set of $\G$. Since $\mid T \mid=\lceil\frac{|E(\G)|}{9}\rceil$, we obtain that the vertex-edge domination set is at most $\lceil\frac{|E(\G)|}{9}\rceil$. This completes the proof.
\wbull

\begin{theorem}
Let $\G =C_n \Box P_2$ be the Cartesian product of $C_n$, $n \geq 4$ and $P_2$. Then, we have
$$\g_{emd}(\G) =
\begin{cases}
\big\lceil \frac{|E(\G)|}{9} \big\rceil +1, & 4\leq n \leq9;\\
\big\lceil \frac{|E(\G)|}{9} \big\rceil, &n \geq 10.
        \end{cases}$$
\end{theorem}
\proof
By Proposition \ref{thm:edimDn}, we obtain that $\B_e(\G)=3$. Similarly, by Proposition \ref{thm:vedominationDn},
the vertex-edge dominant edge metric dimension $\g_{ve}(\G)$ of $\G$ is $\Big\lceil\frac{|E(\G)|}{9} \Big\rceil.$
We divide the proof into a number of cases:

\noindent{\textbf{Case 1:} Let $n=4,5$.

Since $\B_e(\G)=3$ and $\g_{ve}(\G)=1$. The vertex $x_1$ will dominate all the edges but would not resolve all edges
as $\G$ is not a path graph. However, the set $W =\{x_1, x_2, x_3\}$  will dominate and assign unique codes simultaneously to
all the edges $\G$. Thus, we obtain $\g_{emd}(\G)=3$.

\noindent{\textbf{Case 2:} Let $n=6$.

Note that, in this case, $\B_e(\G)=3$ and  the $\g_{ve}(\G)=2$. The set $W =\{x_1, x_3, x_5\}$ will dominate and assign unique codes simultaneously to
all the edges $\G$. Thus, we obtain $\g_{emd}(\G)=3$.

\noindent{\textbf{Case 3:} Let $n=7,8,9$.

Note that, in this case, $\B_e(\G)=3$ and  the $\g_{ve}(\G)=3$. However, neither the edge resolving set by Proposition \ref{thm:edimDn}
nor the vertex-edge dominating set by Proposition \ref{thm:vedominationDn} will dominate all edges $\G$, as some of the edges would
have same codes. The set $T = \{x_i, x_{i+6}, y_{i+3}\}$ where $1 \leq i \leq n$ and $x_{n+i}= x_i, y_{n+i}= y_i$  is vertex-edge dominating set and the set $A= \{x_i, x_{i+1}, x_{i+l}\}$ where $l = \lceil\frac{n}{2}\rceil$ is an edge resolving set for $\G$.
For $n=7$, the set $T$ is vertex-edge dominating set but is not edge resolving set as $r(y_iy_{i+1} | T)= r(y_{i+1}x_{i+1} | T)= r(x_{i+1}x_{i+2} | T)= (1,2,2)$ and  $r(y_{i+1}y_{i+2} | T)= r(y_{i+2}x_{i+2} | T)= r(x_{i+2}x_{i+3} | T)= (2,3,1)$. The set $A$ is an edge resolving set but is not vertex-edge dominant set as the edges $y_{i+2}y_{i+3}$ and $y_{i+l+1}y_{i+l+2}$ are not dominated by any vertex of $A$.
For $n=8$, the set $T$  is vertex-edge dominating set but is not edge resolving set as $r(x_{i+2}y_{i+2} | T)= r(x_{i+2}x_{i+3} | T)= (2,4,1)$, $r(x_{i+3}y_{i+3} | T)= r(y_{i+2}y_{i+3} | T)= (3,3,0)$ and  $r(x_{i+6}y_{i+6} | T)= r(x_{i+7}y_{i+6} | T)= (2,0,3)$. The set $A$ is an edge resolving set but is not vertex-edge dominant set as the edges $y_{i+2}y_{i+3}$, $y_{i+l+1}y_{i+l+2}$, $y_{i+l+2}y_{i+l+3}$ and $y_{i+l+2}x_{i+l+2}$ are not dominated by any  vertex of $A$.
For $n=9$, the set $T$ is vertex-edge dominating set but is not edge resolving set as $r(y_iy_{i+1} | T)= r(y_{i+1}x_{i+1} | T)= r(x_{i+1}x_{i+2} | T)= (1,4,2)$ and  $r(x_{l+i-1}x_{l+i} | T)= r(x_{l+i}y_{l+i} | T)= r(y_{l+i}y_{l+i+1} | T)= (4,1,2)$. The set $A$ is an edge resolving set but is not vertex-edge dominant set as the edges $y_{i+2}y_{i+3}$ and $y_{i+l+1}y_{i+l+2}$. The set $A$ is an edge resolving set but is not vertex-edge dominant set as the edges $y_{i+2}y_{i+3}$, $y_{i+3}y_{i+4}$, $x_{i+3}y_{i+3}$, $y_{i+l+1}y_{i+l+2}$ and $y_{i+l+2}y_{i+l+3}$ are not dominated by any  vertex of $A$.

Hence, the set $W =\{x_1, x_3, x_5, x_7 \}$ will dominate and assign unique codes simultaneously to all the
edges $\G$. Thus, we obtain $\g_{emd}(\G)=4$.

\noindent{\textbf{Case 4:} Let $n\geq10$.

By Proposition \ref{thm:edimDn}, the set $A=\{x_1, x_2, x_{l+1}\}$ is an edge resolving set. Moreover, by Proposition \ref{thm:vedominationDn},
we have $\g_{ve}(\G) = \Big\lceil \frac{|E(\G)|}{9} \Big\rceil$. The corresponding vertex-edge resolving set is also an edge resolving
set of $\G$. Therefore, we obtain $\g_{emd}(\G) = \g_{ve}(\G) =\Big\lceil \frac{|E(C_n \Box P_2)|}{9}\Big\rceil$.
\wbull

\section{Some general results and bounds}\label{sec5}
In this section, we provide some results for general graphs. First, we define some necessary terminologies.
For a graph $\G$, two vertices $x,y\in V(\G)$ are said to be true twins (resp. false twins) if $N[x]=N[y]$ (resp. $N(x)=N(y)$).
In a similar fashion, a vertex $x\in V(\G)$ is called true twin (resp. false twin) in $\G$, if there exists a vertex $y\neq x$
such that $x,~y$ are true twins (resp. false twins).

\begin{lemma}
Let $\G$ be a connected graph. Then there exists a minimum vertex-edge dominating set for $\G$ which
does not have any pair of false twin vertices.
\end{lemma}
\proof
Let $T$ be a minimum vertex-edge dominating set for $\G$ with minimum number of false twin pairs of vertices
and $u,v$ be any arbitrary false twin pair in $T$. Since $u$ and $v$ dominate the same edges in $\G$, by
definition of vertex-edge domination, $T \setminus \{u\}$ and $T \setminus \{v\}$ are vertex-edge dominating
sets in $\G$. This implies that there exists a minimum vertex-edge dominating set for $\G$ which does
not have any pair of false twin vertices. This shows the lemma.
\wbull

\begin{theorem}
For every connected graph $\G$ of order $n$, we have $\B_e(\G) \leq n- \g_{ve}(\G)$. In particular,
if $T$ is a minimum vertex-edge dominating set for $\G$ with no false twin pair of vertices, then
$V (\G) \setminus T$ is an edge resolving set for $\G$.
\end{theorem}
\proof
By Lemma 1, $\G$ has a minimum vertex-edge dominating set $S$ with no pair of false twin vertices. Suppose,
on the contrary, that $V (\G) \setminus T$ is not an edge resolving set for $\G$. By definition of edge
metric dimension, there exist edges $e$ and $f$ such that $r(e | V(\G) \setminus T) = r(f | V(\G) \setminus T)$
if edge $e$ is incident to vertex $u$ and edge $f$ is incident to vertex $v$, then
$r(u | V(\G) \setminus T) = r(v | V(\G) \setminus T)$. This implies that all neighboring vertices of
$u$ and $v$ in $V (\G) \setminus T$ are the same. So, vertices $u$ and $v$ have no neighbor
in the set $T$. On the other hand, we can remove one of the vertices $u$ and $v$ from the set $T$ and get a dominating
set with order of $\mid T\mid-1$. Hence, $u$ and $v$ are false twin vertices, which is a contradiction. Thus, $V (G) \setminus T$
is a resolving set for $\G$. Accordingly, $\B_e(\G) \leq n-\g_{ve}(\G)$.
\wbull

\begin{lemma}
Let $\Gamma$ be a connected graph. If $\Gamma$ has no vertex-edge dominant edge resolving set with cardinality $l$, then $K \subseteq V(\Gamma)$ with $|K| < l$ is not a vertex-edge dominant edge resolving set.
\end{lemma}
\begin{proof}
Let $\Gamma$ be a connected graph. Assume that $\Gamma$ has no vertex-edge dominant edge resolving set with cardinality $l$ and suppose that there exist a vertex-edge dominant edge resolving set $X \subseteq V(\Gamma)$ with $|X| < l$. So, that for every $e_i, e_j \in E(\Gamma)$, we have $r(e_i | X) \neq r(e_j | X)$ and every edge $e \in E(\Gamma)$, is dominated by any vertex of set $X$. Moreover, there exist a set $Y$ and  $Y \subseteq V(\Gamma) \setminus X$ such that $| X \cup Y | = l$. The set $ X \cup Y $ is an edge resolving set for $\Gamma$ as $X$ is an edge resolving set for $\Gamma$ and $ X \cup Y $ is vertex-edge dominant set for $\Gamma$ as $X$ is vertex-edge dominant set for $\Gamma$. So, that the set $ X \cup Y $ is vertex-edge dominant edge resolving set for $\Gamma$, which is contradiction and the proof is complete.
\end{proof}
\wbull

The following proposition bounds the vertex-edge domination number of a connected graph $\G$.
\begin{proposition}
Let $\G$ be a connected graph of order $n$ and size $m$. Then
$$\max\{ \g_{ve}(\G), \B_{e}(\G)\} \leq \g_{emd}(\G) \leq \g_{ve}(\G) + \B_{e}(\G).$$
\end{proposition}
\begin{proof}
Assume that the set $T \subseteq V(\G)$ is
vertex-edge dominating set and the set $A\subseteq V(\G)$ is an edge resolving set of a graph $\G$.
If $T \cup A = T$ or $T \cup A = A$ then $\g_{emd}(\G) = \max\{ \g_{ve}(\G), \B_{e}(\G)\}$.
If $T\cap A \neq \phi$ and $T \subset T \cup A $ or $A \subset T \cup A$, then $\g_{emd}(\G) > \max\{ \g_{ve}(\G), \B_{e}(\G)\}$.
Since, the vertex-edge dominant edge metric dimension of $\G$ is greater than its vertex-edge
domination number and its edge metric dimension, we obtain that $\g_{emd}(\G) \geq \max\{ \g_{ve}(\G), \B_{e}(\G)\}$. If
the vertex-edge dominating set and the edge resolving set of $\G$ do not intersect, i.e. $T \cap A = \phi$, then
$\g_{emd}(\G) \leq \g_{ve}(\G) + \B_{e}(\G)$. This completes the proof.
\end{proof}
\wbull

Next theorem presents sharp upper and lower bounds on the $\g_{emd}(\G)$ of $\G$ in terms of its order.
\begin{theorem}
Let $\G$ be a graph of order $n$. Then, the vertex-edge dominant edge metric dimension of $\G$ satisfies
$$\Big\lfloor \frac{n+2}{4} \Big\rfloor \leq \g_{emd}(\G) \leq n-1.$$
\end{theorem}
\begin{proof}
Let $\G$ be a graph of order $n$ and size $m$. Assume $\Delta(\G)$ is the maximum degree of $\G$ and $\G$ satisfies $2 \leq \Delta(\G) \leq n-1$.
Since $\Delta(\G)=2$ implies that $\G$ is either path $P_n$ or cycle $C_n$. We know that $\B_e(P_n)=1$ and $\g_{ve}(P_n)= \big\lfloor \frac{n+2}{4} \big\rfloor $.
This implies that, we obtain $ \g_{emd}(P_n)=\big\lfloor \frac{n+2}{4}\big\rfloor$. Moreover, we know that $\B_e(C_n)=2$ and
$\g_{ve}(C_n)=\big\lfloor \frac{n+3}{4}\big\rfloor$. We obtain $ \g_{emd}(C_n)=\big\lfloor \frac{n+3}{4}\big\rfloor$.
If maximum degree of $\G$ is $\Delta(\G) > 2$, then $\g_{emd}(\G) > \lfloor \frac{n+2}{4} \rfloor$.
Moreover, if the maximum degree of $\G$ is $n-1$ i.e. $\Delta(\G) = n-1$ then, either $\B_e(\G)=n-1$ or $\B_e(\G)=n-2$
by Proposition \ref{thm:edimn-1} but $\g_{ve}(\G) = 1$ because $\mathrm{diam}(\G)\leq2$. In this case, the edge resolving set
will dominate all edges and we obtained that $\g_{emd}(\G) \leq n-1$. Proved that
$$\Big\lfloor \frac{n+2}{4}\Big\rfloor\leq\g_{emd}(\G)\leq n-1.$$
\end{proof}
\wbull

A vertex $v\in V(\G)$  of an $n$-vertex graph $\G$ is said to be universal, if $d_v=n-1$. The following theorem computes
the vertex-edge dominant edge metric dimension of $\G$, if it has at least one universal vertex.
\begin{theorem}\label{thm:theorem2}
Let $\G$ be a graph of order $n$ and $ \Delta(\G) = n-1$. Then, the vertex-edge dominant edge metric dimension $\g_{emd}(\G)$ of $\G$ is
$\g_{emd}(\G)= n-1 ~\mathrm{or} ~ n-2.$
\end{theorem}
\begin{proof}
Let $\G$ be a graph of order $n$ and size $m$. Suppose that, $\Delta(\G)$ is maximum degree of $\G$ and $\Delta(\G) = n-1$.
By Proposition \ref{thm:edimn-1}, if there exists a vertex $v \in V(\G)$ of degree $n-1$, then,
either $\B_e(\G)=n-1$ or $\B_e(\G)=n-2$. By Corollary \ref{thm:collary}, if $\B_e(\G) = n-1$, then $\mathrm{diam}(\G) = 2$ and
every edge belongs to a triangle. Assume that the $\B_e(\G)=n-2$, then for any $v_i \neq v_j$ with $1\leq i<j\leq n$, there is a vertex $u \in V(\G)$
such that $v_iu \in E(\G)$ and $v_ju \in E(\G)$, so $\mathrm{diam}(\G) = 2$. If $\mathrm{diam}(\G)\leq 2$, then Proposition \ref{thm:propo1}
implies that the vertex-edge domination number is $1$ i.e. $\g_{ve}(\G)=1$. As the edge metric dimension of graph $\G$ is either $\B_e(\G)=n-1$ or
$\B_e(\G)=n-2$ the vertex-edge domination is $1$ i.e. $\g_{ve}(\G) = 1$. This implies that the edge resolving set of $\G$ will dominate and resolve
all the edges of graph $\G$ and we obtain that the vertex-edge dominant edge dimension of $\G$ is $\g_{emd}(\G)= n-1 ~\mathrm{or} ~ n-2.$ This shows the result.
\end{proof}
\wbull

The next theorem refines Theorem \ref{thm:theorem2}.
\begin{theorem}
Let $\G$ be an $n$-vertex graph such that $ \Delta(\G) = n-1$. Assume that there exist at least two vertices with degree $n-1$.
Then, $\g_{emd}(\G) = n-1$.
\end{theorem}
\begin{proof}
Let $\G$ be a graph of order $n$, size $m$, maximum degree $\Delta(\G) = n-1$. If there exist at least two vertices with degree $n-1$,
then, Lemma \ref{thm:lemma} implies that the $\B_e(\G) = n-1$. Moreover, by Theorem \ref{thm:theorem}, we have $\B_e(\G) = n-1$
if and only if for any distinct $v_1,v_2 \in V(\G)$, there exists $u \in V(\G)$ such that $v_1u \in E(\G), v_2u \in E(\G)$ and
$u$ is adjacent to all non-mutual neighbors of $v_1, v_2$. This property restricts $\G$ to have $\mathrm{diam}(\G) \leq 2$.
By Corollary \ref{thm:collary} if $\B_e(\G) = n-1$, then $\mathrm{diam}(\G) = 2$ and every edge lies on a cycle of length 3.
If $\mathrm{diam}(\G) \leq 2$, then $\g_{ve}(\G) = 1$ by Proposition \ref{thm:propo1}. The edge metric dimension of graph $\G$ is
$\B_e(\G)=n-1$ and the vertex-edge domination is $1$ i.e. $\g_{ve}(\G) = 1$. This implies that the edge resolving set of
cardinality $n-1$ will dominate and resolve all the edges of $\G$ and we obtain that the vertex-edge dominant edge dimension of
$\G$ is $\g_{emd}(\G)= n-1$. This shows the result.
\end{proof}
\wbull

Next, we show some result on the vertex-edge dominant edge metric dimension of trees.
The following proposition shows upper and lower bounds on the vertex-edge domination number
of trees.
\begin{proposition}\label{vegammaT}
Let $T_n$ be a tree of order $n$, size $m$, maximum degree $\Delta(T_n)$, and minimum degree $\delta(T_n)$. Then,
$$\Big\lceil \frac{m}{4}\Big\rceil \leq \g_{ve}(T_n) \leq n-\Delta(T_n).$$
\end{proposition}
\begin{proof}
Note that $2\leq\Delta(T_n)\leq n-1$ and $\delta(T_n)=1$. As we know that, $\big\lceil \frac{m}{(\Delta(T_n))^2} \big\rceil \leq \g_{ve}(T_n)$
by Proposition \ref{thm:propo2}. It can easily be verified that $\lceil \frac{m}{(\Delta(T_n))^2}  \rceil \leq \g_{ve}(T_n)$ holds for every tree $T_n$.
By Proposition \ref{thm:propo3}, we know that
$$\g_{ve}(T_n) \leq m-\Delta(T_n)-\frac{\Delta(T_n)(\delta(T_n)-1)}{2}+1.$$
Thus, we obtain $\g_{ve}(T_n) \leq n- \Delta(T_n)$ satisfied by $T_n$. We obtain that
$$\Big\lceil\frac{m}{4}\Big\rceil\leq\g_{ve}(T_n)\leq n-\Delta(T_n).$$
This shows the proposition.
\end{proof}
\wbull

Based on Proposition \ref{vegammaT}, we present upper and lower bounds on the vertex-edge
dominant edge metric dimension of trees.
\begin{theorem}
Let $T_n$ be a tree with order $n\geq2$ and size $m$. Then, the vertex-edge dominant edge metric dimension of $T_n$ satisfies
$$\Big\lceil\frac{m}{4}\Big\rceil \leq \g_{emd}(T_n) \leq n-1.$$
\end{theorem}
\begin{proof}
Note that $T_n$ satisfies $2 \leq \Delta(T_n) \leq n-1$. If maximum degree of tree $T_n$ is $\Delta(T_n) =2$, then
Proposition \ref{thm:propo2} implies that $\big\lceil \frac{m}{4}\big\rceil \leq \g_{emd}(T_n)$.
If maximum degree of $T_n$ is $\Delta(T_n) =n-1$, then Theorem \ref{thm:theorem2}
implies that the vertex-edge dominant edge metric dimension $\g_{emd}(T_n)$ of $T_n$ is
$\g_{emd}(T_n)=n-1$ or $\g_{emd}(T_n)=n-2$. Thus, if the maximum degree of tree $T_n$ satisfies
$2 < \Delta(T_n) < n-1$, then, this implies that the vertex-edge dominant edge metric dimension of
$T_n$ satisfies $\Big\lceil\frac{m}{4}\Big\rceil\leq\g_{emd}(T_n)\leq n-1.$ This shows the theorem.
\end{proof}
\wbull

\subsection{The Corona Product of two graphs}
Let $\Gamma$ and $\Omega$ be two graphs of order $p$ and $q$, respectively. The corona product graph $\Gamma \odot \Omega$
is defined as the graph obtained from $\Gamma$ and $\Omega$, by taking one copy of $\Gamma$ and $p$ copies of $\Omega$ and
joining by an edge the $i^{th}$-vertex of $\Gamma$ with every vertex from the $i^{th}$-copy of $\Omega$. Given a vertex
$u \in V (\Gamma)$, the copy of $\Omega$ whose vertices are adjacent to $u$ is denoted by $\Omega_u$.

\begin{proposition}
Let $\Gamma$ and $\Omega$ be two connected graphs  and $| V(\Omega) | \geq 2$. Then,
$\gamma_{ve}(\Gamma \odot \Omega) = |V (\Gamma)|$. \label{pro:corona}
\end{proposition}
\begin{proof}
Let the graph $\Gamma \odot \Omega$ is corona product of two graphs $\Gamma$ and $\Omega$ and $| V(\Omega) | \geq 2$. If the vertex $x_i \in V(\Gamma \odot \Omega)$, then all edges incident to $N(x_i)$ are dominated  by the vertex $x_i$. The subset $V(\Omega_{x_i}) \subset N(x_i)$ is the vertex set of $i^{th}$-copy of $\Omega$ and  $U = \{y \in V(\Gamma) : x_iy \in E(\Gamma) \}  \subset N(x_i)$. The set $N(x_i) = V(\Omega_{x_i}) \cup U$ shows that the edges of $i^{th}$-copy of $\Omega$ are dominated by the vertex $x_i$. So, that the set $T = V(\Gamma)$ will dominate all edges of $\Gamma \odot \Omega$ and it is clear that $\gamma_{ve}(\Gamma \odot \Omega) \leq |V (\Gamma)|$.
The set $T = V(\Gamma) = \{x_1, x_2, \ldots, x_p\}$ is the minimum vertex-edge dominant set for $\Gamma \odot \Omega$. If the vertex $x_i \not\in T$, then all edges of $i^{th}$-copy of $\Omega$ are not dominated by any vertex of $T$, it is clear that $\gamma_{ve}(\Gamma \odot \Omega) \geq |V (\Gamma)|$. This implies that $\gamma_{ve}(\Gamma \odot \Omega) = |V (\Gamma)|$.
\end{proof}
\wbull

\begin{theorem}
Let $\Gamma$ and $\Omega$ be two connected graphs  and $| V(\Omega) | \geq 2$. Then,
$\gamma_{emd}(\Gamma \odot \Omega) = \beta_e (\Gamma \odot \Omega)$.
\end{theorem}
\begin{proof}
Note that the graph $\Gamma \odot \Omega$ is corona product of two graphs $\Gamma$ with order $p$ and $\Omega$ with order $q$, where $| V(\Omega) | \geq 2$ and  the vertex set of $\Gamma \odot \Omega$ is defined as $V(\Gamma \odot \Omega) = \{x_i | 1\leq i \leq p\} \cup \{y_i^j | 1\leq i \leq p, 1\leq j \leq q\}$. By Theorem \ref{thm:corona} (resp. Proposition \ref{pro:corona}), we have that $\beta_e (\Gamma \odot \Omega) = |V (\Gamma)| \cdot (|V (\Omega)| - 1)$ (resp. $\gamma_{ve}(\Gamma \odot \Omega) = |V (\Gamma)|$). The set $W = \{x_i | 1\leq i \leq p\} \cup \{y_i^j | 1\leq i \leq p, 1\leq j \leq q-1\}$ is an edge resolving set for $\Gamma \odot \Omega$ and $T = \{x_i | 1\leq i \leq p\} \subset W$ is vertex-edge dominant set for $\Gamma \odot \Omega$ by Proposition \ref{pro:corona}. Since, $| W | =  |V (\Gamma)| \cdot (|V (\Omega)| - 1)$ and the set $W$ is vertex-edge dominant edge resolving set for $\Gamma \odot \Omega$.  We obtain the desired result.
\end{proof}
\wbull

\subsection{The Join of two graphs}
Noted that the $\Gamma$ and $\Omega$ are two graphs, the join graph $\Gamma \vee \Omega$ is obtained from $\Gamma$ and $\Omega$ by adding an edge between every vertex of $\Gamma$ and every vertex of $\Omega$. It is clearly observed that the join graph $\Gamma \vee \Omega$ is always a connected independently of the connectivity of the graphs $\Gamma$ and $\Omega$.

\begin{proposition}
Let $\Gamma$ and $\Omega$ be two connected graphs. Then,
$\gamma_{ve}(\Gamma \vee \Omega) = 1 ~or~ 2$. \label{pro:join}
\end{proposition}
\begin{proof}
Let the graph $\Gamma \vee \Omega$ is join of two connected graphs $\Gamma$ and $\Omega$. The vertex set and edge set of $\Gamma \vee \Omega$ are $V(\Gamma \vee \Omega) = V(\Gamma) \cup V(\Omega)$ and $E(\Gamma \vee \Omega) = E(\Gamma) \cup E(\Omega) \cup \{xu | x \in V(\Gamma), u \in V(\Omega) \}$, respectively. If the vertex $x \in V(\Gamma \vee \Omega)$, then all edges incident to $N(x)$ are dominated  by the vertex $x$.

\noindent{\textbf{Case 1:}} Assume That $\Gamma \cong K_n$  or $\Omega \cong K_n$.

For $\Gamma \vee \Omega$, the vertex $x \in V(\Gamma)$ will dominate all edges of $\Gamma \vee \Omega$ as $N(x) = V(\Gamma) \cup V(\Omega)$.  This implies that the singleton set $T = \{x : x \in V(\Gamma)\}$ is vertex-edge dominating set for $\Gamma \vee \Omega$ where $\Gamma \cong K_n$. Therefore, we obtain $\gamma_{ve}(\Gamma \vee \Omega) = 1$.\\
For $\Gamma \vee \Omega$, the vertex $u \in V(\Omega)$ will dominate all edges of $\Gamma \vee \Omega$ as $N(u) = V(\Gamma) \cup V(\Omega)$.  This implies that the singleton set $T = \{u : u \in V(\Omega)\}$ is vertex-edge dominating set for $\Gamma \vee \Omega$ where $\Gamma \cong K_n$. Therefore, we obtain $\gamma_{ve}(\Gamma \vee \Omega) = 1$.

\noindent{\textbf{Case 2:}} Assume That $\Gamma , \Omega \not\cong K_n$.

The set $T = \{x, u : x \in V(\Gamma), u \in V(\Omega) \}$ provides the minimum vertex-edge dominant set of $\Gamma \vee \Omega$. It is clear that the $\gamma_{ve}(\Gamma \vee \Omega) \leq 2$. The vertex $x \in V(\Gamma)$ will dominate all edges incident to $N(x) = V(\Omega) \cup \{y \in V(\Gamma) : xy \in E(\Gamma) \}$. Similarly, the vertex $u \in V(\Omega)$ will dominate all edges incident to $N(u) = V(\Gamma) \cup \{v \in V(\Omega) : uv \in E(\Gamma) \}$. If set $T_1= \{x\}$ (resp. $T_2=\{u\}$), then all edges of graph $\Gamma$ (resp. $\Omega$) are not dominated by $T_1$ (resp. $T_2$), it is clear that $\gamma_{ve}(\Gamma \vee \Omega) \geq 2$. This implies that $\gamma_{ve}(\Gamma \vee \Omega) = 2$. This completes the proof.
\end{proof}
\wbull

\begin{theorem}
Let $\Gamma$ and $\Omega$ be two connected graphs. Then,
$\gamma_{emd}(\Gamma \vee \Omega) = \beta_e (\Gamma \vee \Omega)$.
\end{theorem}
\begin{proof}
Note that the graph $\Gamma \vee \Omega$ is join of two graphs $\Gamma$ with order $p$ and $\Omega$ with order $q$  the vertex set of $\Gamma \vee \Omega$ is defined as $V(\Gamma \vee \Omega) = \{x_i | 1\leq i \leq p\} \cup \{u_j | 1\leq j \leq q\}$. By Theorem \ref{thm:join} (resp. Proposition \ref{pro:join}), we have that $\beta_e (\Gamma \vee \Omega) = |V (\Gamma)| + |V (\Omega)| - 1$ (resp. $\gamma_{ve}(\Gamma \vee \Omega) = 1 ~or~ 2$). The set $W = \{x_i | 1\leq i \leq p\} \cup \{y_j | 1\leq j \leq q-1\}$ is an edge resolving set for $\Gamma \vee \Omega$ and $T = \{x, u\} \subset W$ is vertex-edge dominant set for $\Gamma \vee \Omega$ by Proposition \ref{pro:join}. Since, $| W | =  |V (\Gamma)| + |V (\Omega)| - 1$ and the set $W$ is vertex-edge dominant edge resolving set for $\Gamma \vee \Omega$.  We obtain the desired result.
\end{proof}
\wbull

\section{Comparison between the dominant metric dimension and the vertex-edge dominant edge metric dimension}\label{sec6}
Kelenc et al. \cite{kele} showed the metric dimension and the edge metric dimension of a graph $\G$ are not
comparable, since there exist graphs satisfying one of either $\B(\G)<\B_{e}(\G)$ or $\B(\G)=\B_{e}(\G)$ or $\B(\G)>\B_{e}(\G)$.
In response, Knor et al. \cite{Knor} presented further infinite families of graphs satisfying $\B(\G)>\B_{e}(\G)$.
Lewis \cite{Lew} in his PhD thesis showed that, for any graph $\G$, the inequality $\g_{ve}(\G) \leq \g(\G)$
holds, in general.

In view of the comparison by Kelenc et al. \cite{kele}, it is natural to consider a similar comparison between
the dominant metric dimension and the vertex-edge dominant edge metric dimension of graphs. This section presents
families of graph showing that these aforementioned parameters are also not comparable, in general.
Table \ref{tab:COM} presents standard families of graphs with their dominant metric dimension$ \g_{md}(\G)$ and the
vertex-edge dominant edge metric dimension $\g_{emd}(\G)$. Note that some standard families such the path graph $P_n$,
the cycle graph $C_n$, the complete graph $K_n$ and the complete bipartite graph $K_{n,m}$
showcase the behavior $\g_{md}(\G) \geq \g_{emd}(\G)$. For instance, the star graph $S_{1,n}$ exhibits
$\g_{md}(S_{1,n}) > \g_{emd}(S_{1,n})$, whereas, the wheel and fan graphs satisfy $\g_{md}(\G) < \g_{emd}(\G)$.
Moreover, some path graphs satisfy the $\g_{md}(\G) = \g_{emd}(\G)$ instead.

\begin{table}[h]
		\centering
		\begin{tabular}{|c|c|c|c|c|c|c|c|c|c|}
		\hline
	    Families of graphs &   $\g_{md}(\G)$ &  $\g_{emd}(\G)$& Remark \\ \hline
		$P_2$&$\B(P_2)=1$& $\B_e(P_2)=1$& $\g_{md}(P_2) = \g_{emd}(P_2)$\\
		$P_3$&$\B(P_3) + 1=2 $& $\B_e(P_3)=1$& $\g_{md}(P_3) > \g_{emd}(P_3)$\\
		$P_4$&$\B(P_4) + 1=2 $ & $\B_e(P_4) + 1=2$& $\g_{md}(P_4) = \g_{emd}(P_4)$\\
		$P_5$&$\g(P_5)= \lceil \frac{n}{3} \rceil$=2 & $\B_e(P_5) + 1=2$& $\g_{md}(P_5) = \g_{emd}(P_5)$\\
        $P_n : n \geq 6$& $\g(P_n)= \lceil \frac{n}{3} \rceil$& $\g_{ve}(P_n) = \lfloor \frac{n+2}{4} \rfloor$& $\g_{md}(P_n) \geq \g_{emd}(P_n)$\\ \hline
        $C_n : n=3,4,5$&$\B(C_n)=2$& $\B_e(C_n)=2$& $\g_{md}(C_n) = \g_{emd}(C_n)$\\
        $C_n : n=6,7$&$\B(C_n)+1=3$ & $\B_e(C_n)=2$& $\g_{md}(C_n) > \g_{emd}(C_n)$\\
        $C_n : n=8$&$\B(C_8)+1=3$ & $\B_e(C_8)+1=3$& $\g_{md}(C_8) = \g_{emd}(C_8)$\\
        $C_n : n \geq 9$& $\g(C_n)= \lceil \frac{n}{3} \rceil$& $\g_{ve}(C_n)= \lfloor \frac{n+3}{4} \rfloor $& $\g_{md}(C_n) \geq \g_{emd}(C_n)$\\ \hline
        $S_{1,n} : n \geq 3$& $\B(S_{1,n})+1$& $ \B_{e}(S_{1,n}) = n-1$ & $\g_{md}(S_{1,n}) > \g_{emd}(S_{1,n})$\\ \hline
        $K_n : n \geq 2$& $\B(K_n)$& $ \B_{e}(K_n) $ & $\g_{md}(K_n) = \g_{emd}(K_n)$\\ \hline
        $K_{n,m} : n,m \geq 2$ &$\B(K_{n,m})$ & $ \B_{e}(K_{n,m})$ &$\g_{md}(K_{n,m}) = \g_{emd}(K_{n,m})$ \\ \hline
        $W_{1,n}; n \geq 6$& $\B(W_{1,n})$ & $\B_e(W_{1,n})$& $\g_{md}(W_{1,n}) < \g_{emd}(W_{1,n})$\\
        $n \equiv 0,2,4 (\mod 5)$ &&&\\ \hline
        $W_{1,n}; n \geq 6$& $\B(W_{1,n}) + 1$ & $\B_e(W_{1,n})$& $\g_{md}(W_{1,n}) < \g_{emd}(W_{1,n})$\\
        $n \equiv 1, 3(\mod 5)$ &&&\\ \hline
        $F_{1,n}; n \geq 6$& $\B(F_{1,n})$ & $\B_e(F_{1,n})$& $\g_{md}(F_{1,n}) < \g_{emd}(F_{1,n})$\\
        $n \equiv 0,2,4 (\mod 5)$ &&&\\ \hline
        $F_{1,n}; n \geq 6$& $\B(F_{1,n}) + 1$ & $\B_e(F_{1,n})$& $\g_{md}(F_{1,n}) < \g_{emd}(F_{1,n})$\\
        $n \equiv 1, 3 (\mod 5)$ &&&\\ \hline
        \end{tabular}
        \vspace*{0.3cm}
        \caption{Comparison between the dominant metric dimension and the vertex-edge dominant edge metric dimension of some families of graphs.}\label{tab:COM}
	\end{table}
Based on the data in Table \ref{tab:COM}, there is no comparison between the dominant metric dimension and the vertex-edge dominant edge metric dimension, in general.

Next, we focus on bipartite graphs for comparison between the dominant metric dimension and the vertex-edge dominant edge metric dimension.
\subsection{Comparison between $\g_{md}$ and $\g_{emd}$ for bipartite graphs}

Although, general graphs are not comparable for $\g_{md}$ and $\g_{emd}$, we show that trees are comparable.
\begin{theorem}\label{treescomparison}
Let $T_n$ be a tree. Then, we have $\g_{md}(T_n)\geq \g_{emd}(T_n)$.
\end{theorem}
\proof
By Proposition \ref{prop0}, we have $\g_{ve}(T_n) \leq \g(T_n)$ if $T_n$ not a path.
Moreover, by Remark \ref{keleremark}, we have $\B(T_n) = \B_e(T_n)$.
Let $S \subseteq V(T_n)$ be a dominating set and $B \subseteq V(T_n)$ be a resolving set of $T_n$.
Moreover, let $T \subseteq V(T_n)$ be a vertex-edge dominating set and $A \subseteq V(T_n)$ be an edge resolving set of $T_n$.
We distinguish the following cases:

\noindent{\textbf{Case 1:} If $S \subset B$ and $T \subset A$, then $\g_{md}(T_n) = \g_{emd}(T_n)$.

\noindent{\textbf{Case 2:} If $B \subset S$ and $A \subset T$, then $\g_{md}(T_n) \geq \g_{emd}(T_n)$.

\noindent{\textbf{Case 3:} If $B \nsubseteq S$ and $A \nsubseteq T$, then $\g_{md}(T_n) \geq \g_{emd}(T_n)$.

It is clear that the dominant metric dimension of $T_n$ is always greater or equal to its vertex-edge dominant edge metric dimension,
i.e. $\g_{md}(T_n)\geq \g_{emd}(T_n)$.
\wbull

By Theorem \ref{treescomparison}, trees are comparable with respect to $\g_{md}$ and $\g_{emd}$. However, for non-tree bipartite graphs,
we show that there is no comparison between $\g_{md}$ and $\g_{emd}$.

First, we construct a non-tree bipartite graph $\G$ satisfying $\g_{md}(\G)>\g_{emd}(\G)$.
The graph $\G$ in Figure \ref{fig:gamma} having $V(\G)=\{a,c,e,f,h\}\cup\{b,d,g,i,j\}$ and $E(\G)=\{e_1,e_2,\ldots,e_{14}\}$.
We show that $\g_{md}(\G)>\g_{emd}(\G)$ for $\G$ in Figure \ref{fig:gamma}.

\begin{figure}[htbp!]
\centering
  \includegraphics[width=5cm]{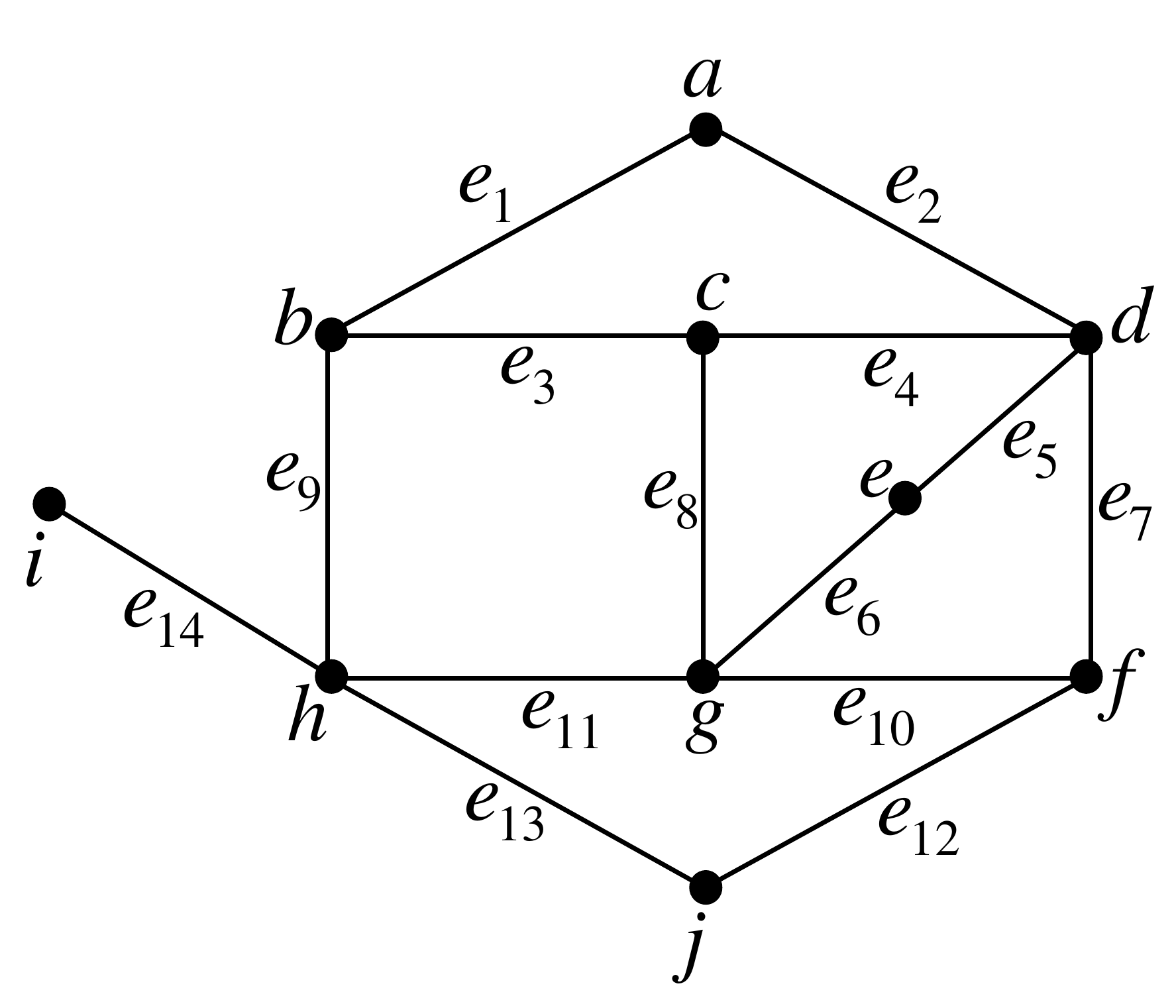}
  \caption{The non-tree bipartite graph $\G$ satisfying $\g_{md}(\G)>\g_{emd}(\G)$.}\label{fig:gamma}
\end{figure}
Next, we show that $\g_{md}(\G)=5$ and $\g_{emd}(\G)=3$, thus, satisfying $\g_{md}(\G)>\g_{emd}(\G)$.
Note that the set $B = \{a,c,e,f \}$ is a resolving set of $\G$. Since, $\G$ is not a path graph, we have $\B(\G)\geq 2$.
The representations of all vertices of $\G$ with respect to the set $B =\{a,c,e,f\}$ are given in the Table \ref{tab:r1}.
\begin{table}[htbp!]
		\centering
        \begin{tabular}{|c|c|c|c|c|c|c|c|c|c|}
		\hline
	    Vertex &$a$ & $b$ & $c$& $d$ &$e$\\ \hline
		${\small r(.,.)}$ & $(0,2,2,2)$& $(1,1,3,3)$& $(2,0,2,2)$& $(1,1,1,1)$& $(2,2,0,2)$\\ \hline
	    Vertex &$f$& $g$&$h$&$i$ &$j$ \\ \hline
		${\small r(.,.)}$ & $(2,2,2,0)$& $(3,1,1,1)$& $(2,2,2,2)$& $(3,3,3,3)$& $(3,3,3,1)$\\ \hline
        \end{tabular}
        \vspace*{0.3cm}
        \caption{The representations of all vertices of $\G$ with respect to the set $B =\{a,c,e,f\}$.}\label{tab:r1}
\end{table}
Table \ref{tab:r1} shows the representations of all the vertices of $\G$ with respect to $B = \{a,c,e,f \}$ are unique.
Thus, we deduce that $\B(\G) \leq 4$. Suppose to the contrary that $\B(\G) = 3$, and no two vertices have the same representation.
Let $B$ be a resolving set in $\G$ having cardinality three. Assume $V(\G) = X \cup Y$ and let $X =\{a,c,e,f,h\}$ and $Y =\{b,d,g,i,j\}$
We divide this into the following cases:

\noindent{\textbf{Case 1:} Assume $B \subset X$ is a resolving set and $|B| = 3$. If $c,e \not\in B$, then $r(c | X \setminus \{c, e\}) = r(e | X \setminus \{c, e\})$. If $e, f \not\in B$, then $r(e | X \setminus \{e, f\}) = r(f | X \setminus \{e, f\})$. It is clear that $c, e, f \in B$ but $B= \{c,e,f\}$ is not a resolving set for $\G$ as $r(a | B) = r(h | B)$, so, any pair of vertices $c,e$ and $f$ is compulsory for resoling set.
If $B_1=\{a,c,f\}$, then $r(e|B_1)=r(h|B_1)$, which is a contradiction. If $B_2 =\{c,f,h\}$, then $r(a|B_2)=r(e|B_2)$, which is a contradiction. If $B_3 =\{a,e,f\}$, then $r(c|B_3)=r(h|B_3$, which is a contradiction. If $B_4 =\{e,f,h\}$, then $r(a|B_4)=r(c|B_4)$, which is a contradiction. If $B_5=\{a,c,e\}$, then $r(h|B_5)=r(f|B_5)$, which is a contradiction. If $B_6 =\{c,e,h\}$, then $r(a|B_6)=r(f|B_6)$, which is a contradiction. This shows that no resolving set $B \subset X$, of cardinality three exists in $\G$.

\noindent{\textbf{Case 2:} Assume $B \subset Y$ is a resolving set and $|B| = 3$. If $b \not\in B$, then $r(c | Y \setminus \{b\}) = r(e | Y \setminus \{b\})$. If $j \not\in B$, then $r(e | Y \setminus \{j\}) = r(f | Y \setminus \{j\})$. It is clear that vertex $b$ and vertex $j$ are compulsory for resoling set.
If $B_1=\{b,d,j\}$, then $r(a|B_1)=r(c|B_1)$, which is a contradiction. If $B_2 =\{b,g,j\}$, then $r(a|B_2)=r(c|B_2)$ and $r(d|B_2)=r(i|B_2)$, which is a contradiction. If $B_3 =\{b,d,i\}$, then $r(a|B_3)=r(c|B_3)$, which is a contradiction. This shows that no resolving set $B \subset Y$, of cardinality three exists in $\G$.

\noindent{\textbf{Case 3:}
Assume $B \subset X \cup Y$ is a resolving set and $|B| = 3$. From the vertices $b, c$ and $e$, a vertex is compulsory for resolving set because $r(c | V(\G) \setminus \{b,c,e\}) = r(e | V(\G) \setminus \{b,c,e\})$. Similarly, From the vertices $e, f$ and $j$, a vertex is compulsory for resolving set because $r(e | V(\G) \setminus \{e,f,j\}) = r(f | V(\G) \setminus \{e,f,j\})$.
If $B_1=\{b,f,x_1\}$, and $x_1 \in \{d,e,g,h,i,j\}$ then $r(a|B_1)=r(c|B_1)$, which is a contradiction. If $B_2=\{b,f,x_2\}$, and $x_2 \in \{a,c,h,i\}$ then $r(g|B_2)=r(j|B_2)$, which is a contradiction. If $B_3=\{b,j,x_3\}$, and $x_3 \in \{d,e,f,g,h,i\}$ then $r(a|B_3)=r(c|B_3)$, which is a contradiction. If $B_4=\{b,f,x_4\}$, and $x_4 \in \{a,c,h\}$ then $r(g|B_4)=r(i|B_4)$, which is a contradiction. If $B_5=\{e,f,x_5\}$, and $x_5 \in \{b,d,h,i,j\}$ then $r(a|B_5)=r(c|B_5)$, which is a contradiction. If $B_6=\{e,f,x_6\}$, and $x_6 \in \{a,g\}$ then $r(c|B_6)=r(h|B_6)$, which is a contradiction. If $B_7=\{c,e,f\}$, then $r(a|B_7)=r(h|B_7)$, which is a contradiction. If $B_8=\{c,f,x_8\}$, and $x_8 \in \{d,g,h,i,j\}$ then $r(a|B_8)=r(e|B_8)$, which is a contradiction. If $B_9=\{c,f,x_9\}$, and $x_9 \in \{b,e,j\}$ then $r(d|B_9)=r(g|B_9)$, which is a contradiction. This shows that no resolving set $B \subset X \cup Y$, of cardinality three exists in $\G$.

This shows that no resolving set of cardinality three exists in $\G$. Thus, we have $\B(\G) \geq 4$ and therefore, $\B(\G) = 4$. Moreover, note that the set $S=\{h,d\}$ is dominating set, as it dominate all the vertices, however, it does not resolve all the vertices of $\G$. The set $A=\{a,c,e,f\}$ is a resolving set as it resolve all the vertices, however, it does not dominate all the vertices. So, the set $D=\{a,c,e,f,h\}$ is dominant resolving set and we have $\g_{md}(\G)=5$.

Next we show that $\g_{emd}(\G)=3$, for $\G$ given in Figure \ref{fig:gamma}.
Note that the set $A = \{b,g,j\}$ is an edge resolving set of $\G$.
The representations of all edges of $\G$ with respect to the set $A$ are given
in the Table \ref{tab:er}

\begin{table}[htbp!]
		\centering
        \begin{tabular}{|c|c|c|c|c|c|c|c|c|c|}
		\hline
	    Edge &$e_1$ & $e_2$ & $e_3$& $e_4$ &$e_5$ &$e_6$ &$e_7$\\ \hline
		${\small r(.,.)}$ & $(0,2,2)$& $(1,2,2)$& $(0,1,2)$& $(1,1,2)$& $(2,1,2)$ & $(2,0,2)$& $(2,1,1)$\\ \hline
	    Edge &$e_8$& $e_9$&$e_{10}$&$e_{11}$ &$e_{12}$ &$e_{13}$ &$e_{14}$\\ \hline
		${\small r(.,.)}$ & $(1,0,2)$& $(0,1,1)$& $(2,0,1)$& $(1,0,1)$& $(2,1,0)$ & $(1,1,0)$& $(1,1,1)$\\ \hline
        \end{tabular}
        \vspace*{0.3cm}
        \caption{The representations of all edges of $\G$ with respect to the set $A$.}\label{tab:er}
\end{table}
Since all the representations in Table \ref{tab:er} are unique, we deduce that $\B_e(\G)\leq3$.
Suppose on the contrary that $\B_e(\G) = 2$. Let $A$ be an edge resolving set of cardinality
two in $\G$ such that no two edges have the same representation. We divide our discussion
into a number of cases:

\noindent{\textbf{Case 1:} Assume $A = \{e,x\}$ and $x \in \{a,d,g,h,i\}$. Then, we have $r(e_4| A) = r(e_7| A)$ and $r(e_8| A) = r(e_{10}| A)$, thus, arising a contradiction.
Assume $A = \{e,y\}$ and $y \in \{b,c,f,j\}$. Then, we have $r(e_4| A) = r(e_8| A)$, $r(e_5| A) = r(e_6| A)$ and $r(e_7| A) = r(e_{10}| A)$, thus, arising a contradiction.

\noindent{\textbf{Case 2:} Assume $A = \{x,y\}$ and $x,y \in \{g,h,i\}$ or $A= \{a,d\}$. Then, we have $r(e_4| A)= r(e_5| A) = r(e_7| A)$ and $(e_6| A) = r(e_8| A) = r(e_{10}| A)$, thus, arising a contradiction. Assume $A= \{c,b\}$. Then, we have $r(e_4| A) = r(e_8| A)$ and $(e_5| A) = r(e_6| A) = r(e_7| A) = r(e_{10}| A)$, thus, arising a contradiction. Assume $A= \{f,j\}$. Then, we have $r(e_7| A) = r(e_{10}| A)$ and $(e_4| A) = r(e_5| A) = r(e_6| A) = r(e_8| A)$, thus, arising a contradiction.

\noindent{\textbf{Case 3:} Assume $A_1 = \{x,y_1\}$ and $x \in \{g,h,i\}$, $y_1 \in \{a,d\}$. Then, we have $r(e_4| A_1)= r(e_5| A_1) = r(e_7| A_1)$ and $(e_6| A_1) = r(e_8| A_1) = r(e_{10}| A_1)$, thus, arising a contradiction.
Assume $A_2 = \{x,y_2\}$ and $x \in \{g,h,i\}$, $y_2 \in \{b,c\}$. Then, we have $r(e_5| A_2) = r(e_7| A_2)$ and $(e_6| A_2) = r(e_{10}| A_2)$, thus, arising a contradiction.
Assume $A_3 = \{x,y_3\}$ and $x \in \{g,h,i\}$, $y_3 \in \{f,j\}$. Then, we have $r(e_4| A_3) = r(e_5| A_3)$ and $(e_6| A_3) = r(e_8| A_3)$, thus, arising a contradiction.

\noindent{\textbf{Case 4:} Assume $A_1 = \{x_1,y_1\}$ and $x_1 \in \{a,d\}$, $y_1 \in \{b,c\}$. Then, we have $r(e_5| A_1) = r(e_7| A_1)$ and $(e_6| A_1) = r(e_{10}| A_1)$, thus, arising a contradiction.
Assume $A_2 = \{x_2,y_2\}$ and $x_2 \in \{a,d\}$, $y_2 \in \{f,j\}$. Then, we have $r(e_4| A_2) = r(e_5| A_2)$ and $(e_6| A_2) = r(e_8| A_2)$, thus, arising a contradiction.
Assume $A_3 = \{x_3,y_3\}$ and $x_3 \in \{b,c\}$, $y_3 \in \{f,j\}$. Then, we have $r(e_4| A_3) = r(e_8| A_3)$, $r(e_5| A_3) = r(e_6| A_3)$ and $(e_7| A_3) = r(e_{10}| A_3)$, thus, arising a contradiction.

Thus, we have $\B_e(\G) \geq 3$, which altogether leads to $\B_e(\G) = 3$.
The set $T =\{h,d\}$ is a vertex-edge dominating set, it dominate all the edges but it does not resolve all the edges of $\G$.
The set $A=\{b,g,j\}$ is an edge resolving set, as it resolves all the edges as well as it dominate all the edges.
So, the set $A =\{b,g,j\}$ is a vertex-edge dominant edge resolving set and we obtain $\g_{emd}(\G)=3$.
\begin{remark}\label{remark01}
The non-tree bipartite graph $\G$ in Figure \ref{fig:gamma} satisfies $\g_{md}(\G)>\g_{emd}(\G)$. More precisely, we have
$\g_{md}(\G)=5$ and $\g_{emd}(\G)=3$.
\end{remark}

Next, we construct a non-tree bipartite graph $\Omega$ satisfying $\g_{md}(\Omega)<\g_{emd}(\Omega)$.
Consider the graph $\Omega$ in Figure \ref{fig:gamma1} having $V(\Omega)=\{a_1,a_2,\ldots,a_8\}\cup\{b_0,b_1,\ldots,b_8\}$ and $E(\Omega)=\{a_lb_0, a_lb_l, a_{l+1}b_l ; 1\leq l \leq 8 \}$.
\begin{figure}[h]
\begin{center}
  \includegraphics[width=5cm]{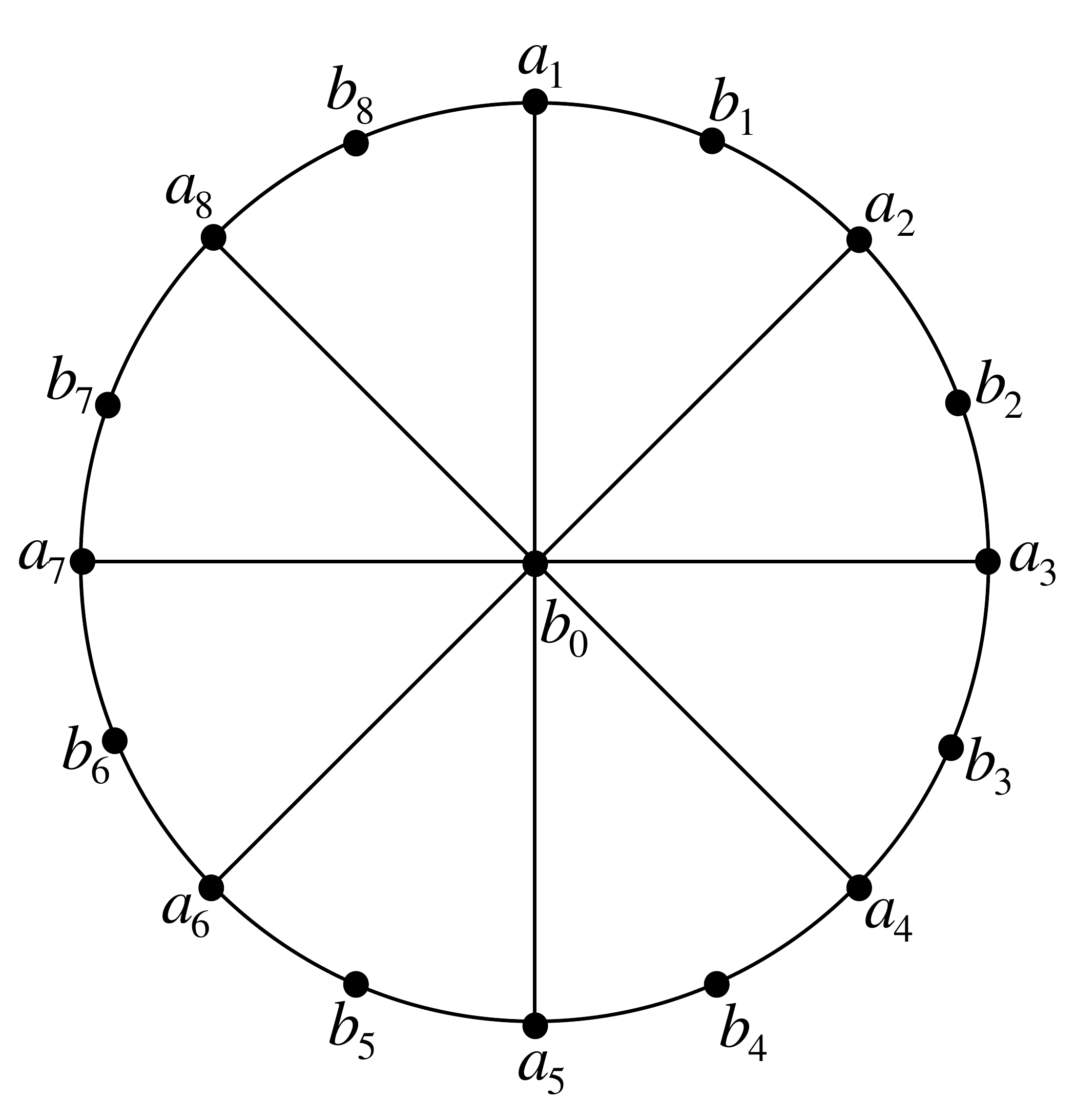}
  \caption{The non-tree bipartite graph $\Omega$ satisfying $\g_{md}(\Omega)<\g_{emd}(\Omega)$.}\label{fig:gamma1}
\end{center}
\end{figure}
Next, we show that $\g_{md}(\Omega)=6$ and $\g_{emd}(\Omega)=7$, thus, satisfying $\g_{md}(\Omega)<\g_{emd}(\Omega)$.
Note that the set $B =\{a_4,b_1,b_2,b_5,b_6\}$ is a resolving set of $\Omega$. Since, $\Omega$ is not a path graph, we have $\B(\Omega)\geq 2$.
The representations of all vertices of $\Omega$ with respect to the set $B =\{a_4,b_1,b_2,b_5,b_6\}$ are given in the Table \ref{tab:r11}.
\begin{table}[h]
		\centering
        \begin{tabular}{|c|c|c|c|c|c|c|}
		\hline
	    Vertices &$a_1$ & $a_2$ & $a_3$& $a_4$ &$a_5$ &$a_6$ \\ \hline
		${\small r(.,.)}$ & $(1,3,2,3,3)$& $(1,1,2,3,3)$& $(3,1,2,3,3)$& $(3,3,0,3,3)$& $(3,3,2,1,3)$ & $(3,3,2,1,1)$\\ \hline
	    Vertices  &$a_7$ &$a_8$&$b_0$&$b_1$& $b_2$&$b_3$\\ \hline
		${\small r(.,.)}$  & $(3,3,2,3,1)$ & $(3,3,2,3,3)$&$(2,2,1,2,2)$& $(0,2,3,4,4)$& $(2,0,3,4,4)$& $(4,2,1,4,4)$ \\ \hline
	    Vertices&$b_4$ &$b_5$ &$b_6$ &$b_7$ &$b_8$ \\ \cline{1-6}
		${\small r(.,.)}$& $(4,4,1,2,4)$& $(4,4,3,0,2)$& $(4,4,3,2,0)$& $(4,4,3,4,2)$& $(2,4,3,4,4)$\\ \cline{1-6}
        \end{tabular}
        \vspace*{0.3cm}
        \caption{The representations of all vertices of $\Omega$ with respect to the set $B$.}\label{tab:r11}
\end{table}
Table \ref{tab:r11} shows the representations of all the vertices of $\Omega $ with respect to $B = \{a_4, b_1, b_2, b_5, b_6 \}$ are unique.
Thus, we deduce that $\B(\Omega)\leq5$. Moreover, the set $B =\{a_4, b_1, b_2, b_5, b_6\}$ is not a dominating set of $\Omega$ because it does not
dominate the vertices $a_8, b_7, b_8$, so, $\gamma(\Omega)=6$.
Furthermore, the set $S=\{a_4, a_8, b_1, b_2, b_5, b_6\}$ is a dominating set, as it dominates all the vertices and it resolves all the vertices
of $\Omega$. So, the set  $D=S=\{a_4, a_8, b_1, b_2, b_5, b_6\}$ is a dominant resolving set and, thus, we have $\g_{md}(\Omega)=6$.

Next, we consider the set $A =\{a_1,a_2,a_3,a_4,a_5,a_6,a_7\}$ which is an edge resolving set of $\Omega$.
The representation of all the edges of $\Omega$ with respect to the set $A =\{a_1, a_2, a_3, a_4, a_5, a_6, a_7\}$
are given in the Table \ref{tab:er1}.
\begin{table}[htbp!]
		\centering
        \begin{tabular}{|c|c|c|c|c|c|}
		\hline
	    Edges &$a_1b_0$ & $a_2b_0$ & $a_3b_0$& $a_4b_0$ &$a_5b_0$\\ \hline
		${\small r(.,.)}$ & $(0,1,1,1,1,1,1)$& $(1,0,1,1,1,1,1)$& $(1,1,0,1,1,1,1)$& $(1,1,1,0,1,1,1)$& $(1,1,1,1,0,1,1)$\\ \hline
	    Edges &$a_6b_0$ &$a_7b_0$&$a_8b_0$& $a_1b_1$&$a_2b_2$\\ \hline
		${\small r(.,.)}$& $(1,1,1,1,1,0,1)$& $(1,1,1,1,1,1,0)$ & $(1,1,1,1,1,1,1)$& $(0,1,2,2,2,2,2)$& $(2,0,1,2,2,2,2)$\\ \hline
	    Edges &$a_3b_3$ &$a_4b_4$ &$a_5b_5$ &$a_6b_6$ &$a_7b_7$\\ \hline
		${\small r(.,.)}$& $(2,2,0,1,2,2,2)$ & $(2,2,2,0,1,2,2)$ & $(2,2,2,2,0,1,2)$ & $(2,2,2,2,2,0,1)$ & $(2,2,2,2,2,2,0)$\\ \hline
	    Edges &$a_8b_8$ &$b_1a_2$ &$b_2a_3$ &$b_3a_4$ &$b_4a_5$\\ \hline
		${\small r(.,.)}$& $(1,2,2,2,2,2,2)$ & $(1,0,2,2,2,2,2)$ & $(2,1,0,2,2,2,2)$ & $(2,2,1,0,2,2,2)$ & $(2,2,2,1,0,2,2)$\\ \hline
	    Edges &$b_5a_6$ &$b_6a_7$ &$b_7a_8$ & $b_8a_1$\\ \cline{1-5}
		${\small r(.,.)}$& $(2,2,2,2,1,0,2)$ & $(2,2,2,2,2,1,0)$ & $(2,2,2,2,2,2,1)$ & $(0,2,2,2,2,2,2)$ \\ \cline{1-5}
        \end{tabular}
        \vspace*{0.3cm}
        \caption{The representations of all edges of $\Omega$ with respect to the set $A$.}\label{tab:er1}
\end{table}
Table \ref{tab:er1} shows the representations of all the edges of $\Omega$ with respect to $A$ are unique. We deduce that $\B_e(\Omega)\leq7$.
Suppose that $\B_e(\Omega) = 6$, such that no two edges have the same representation. Let $A$ be an edge resolving set
of cardinality 6 in $\Omega$. We discuss the following case:

\noindent{\textbf{Case 1:} Assume that $A =\{a_1,a_2,a_3,a_4,a_5,a_6,a_7\}\setminus\{a_k:k=1,2,\ldots,7\}$. Then, $r(a_kb_0| A) = r(a_8b_0| A) = (1,1,1,1,1,1)$
where $k=1,2,\ldots,7$, which contradicts to our assumption.

\noindent This shows that $\B_e(\Omega) \geq 7$, as $A$ can not be an edge resolving set. This altogether leads to $\B_e(\Omega) = 7$.\\

Moreover, the set $T =\{b_0\}$ is a vertex-edge dominating set, as it dominates all the edges but it does not resolve all the edges of $\Omega$.
The set $A =\{a_1, a_2, a_3, a_4, a_5, a_6, a_7\}$ is an edge resolving set, it resolves all the edges and it dominates all the edges as well.
So, the set $A=\{a_1, a_2, a_3, a_4, a_5, a_6, a_7\}$ is a vertex-edge dominant edge resolving set, and, we have $\g_{emd}(\Omega)=7$.
Thus, we have the following remark.
\begin{remark}\label{remark02}
The non-tree bipartite graph $\Omega$ in Figure \ref{fig:gamma1} satisfies $\g_{md}(\Omega)<\g_{emd}(\Omega)$. More precisely, we have
$\g_{md}(\Omega)=6$ and $\g_{emd}(\Omega)=7$.
\end{remark}

Finally, we construct a non-tree bipartite graph $\Pi$ satisfying $\g_{md}(\Pi)=\g_{emd}(\Pi)$.
Consider the graph $\Pi$ in Figure \ref{fig:gamma2} having $V(\Pi)=\{a_1,a_2,a_3,a_4\}\cup\{b_1,b_2\}$ and
$E(\Pi)=\{a_ib_1,a_ib_2:1\leq i\leq4\}$.

\begin{figure}[htbp!]
\centering
  \includegraphics[width=4cm]{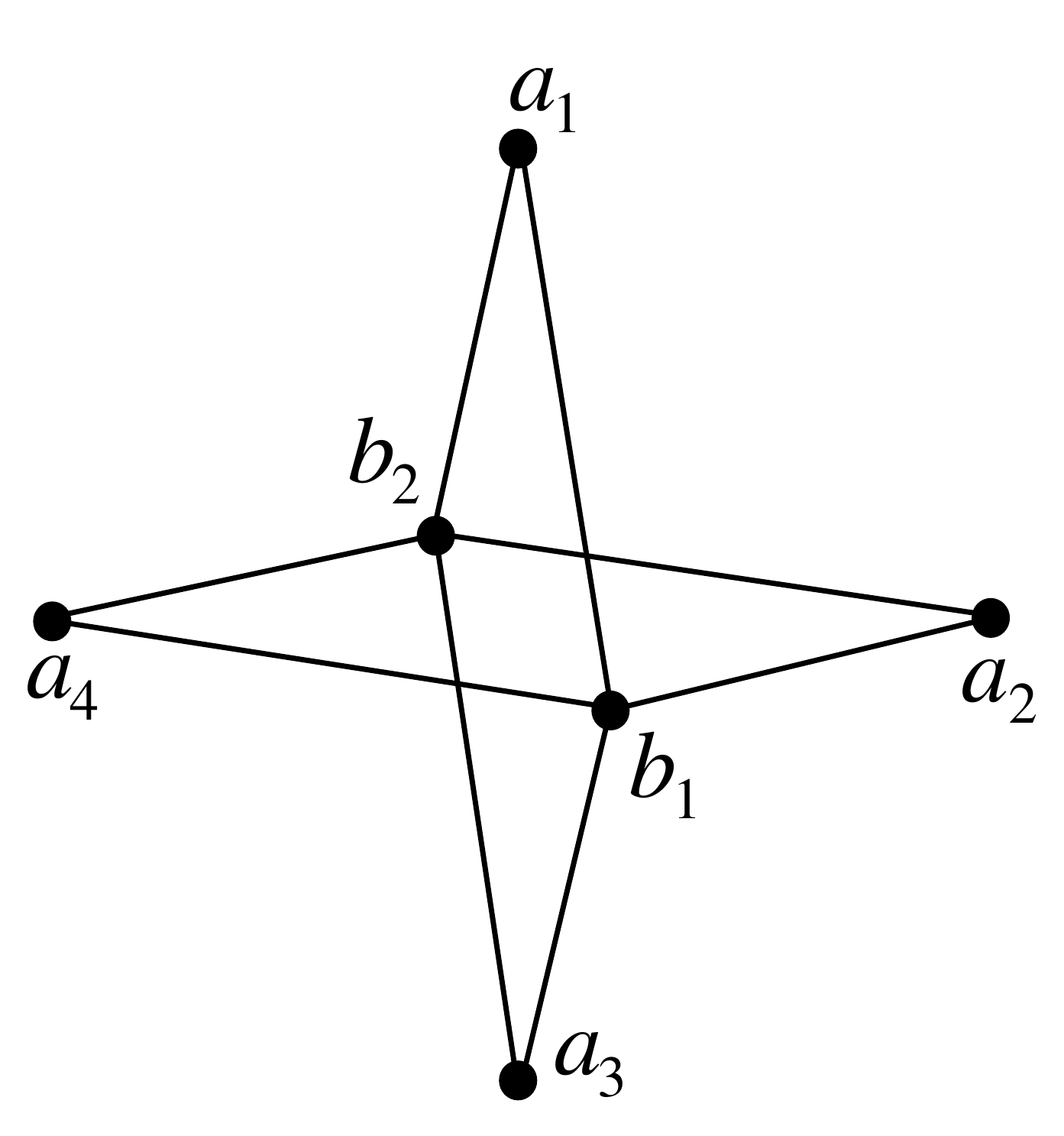}
  \caption{The non-tree bipartite graph $\Pi$ satisfying $\g_{md}(\Pi)=\g_{emd}(\Pi)$.}\label{fig:gamma2}
\end{figure}
Next, we show that $\g_{md}(\Pi)=\g_{emd}(\Pi)=4$, thus, satisfying $\g_{md}(\Pi)=\g_{emd}(\Pi)$.
Note that the set $B =\{a_1,a_2,a_3,b_1\}$ is a resolving set of $\Pi$. Since, $\Pi$ is not a path graph, we have $\B(\Pi)\geq 2$.
The representations of all vertices of $\Pi$ with respect to the set $B =\{a_1, a_2, a_3, b_1\}$ are given in the Table \ref{tab:r12}.

\begin{table}[htbp!]
		\centering
        \begin{tabular}{|c|c|c|c|c|c|c|c|c|c|}
		\hline
	    Vertices &$a_1$ & $a_2$ & $a_3$& $a_4$ &$b_1$ &$b_2$ \\ \hline
		${\small r(.,.)}$ & $(0,2,2,1)$& $(2,0,2,1)$& $(2,2,0,1)$& $(2,2,2,1)$& $(1,1,1,0)$ & $(1,1,1,2)$\\ \hline
        \end{tabular}
        \vspace*{0.3cm}
        \caption{The representations of all vertices of $\Pi$ with respect to the set $B=\{a_1, a_2, a_3, b_1\}$.}\label{tab:r12}
\end{table}
Table \ref{tab:r12} shows the representations of all the vertices of $\Pi$ with respect to $B$ are unique.
Thus, we deduce that $\B(\Pi) \leq 4$. Suppose that $\B(\Pi)=3$, and let $B$ be a resolving set of cardinality 3.
We divide our discussion into a number of cases:

\noindent{\textbf{Case 1:} Assume that $B=\{a_1, a_2, a_3\}$. Then, $r(b_1| B) = r(b_2| B) = (1,1,1)$, which is contradicts to our assumption.

\noindent{\textbf{Case 2:} Assume that $B=\{a_1, a_2, a_3, b_1\} \setminus \{ a_p \}$ where $p=1,2,3$. Then, we have $r(a_p| B) = r(a_4| B) = (2,2,1)$,
which is contradicts to our assumption.

Since $B$ could not be a resolving set, we obtain that $\B(\Pi) \geq 4$, which then leads to $\B(\Pi) = 4$.
The set $S=\{a_1,b_1\}$ is dominating set, as it dominates all the vertices but it does not resolve all the vertices of $\Pi$.
The set $B =\{a_1, a_2, a_3, b_1\}$ is a resolving set, as it resolve all the vertices as well as it dominates all the vertices.
So, the set $B=\{a_1, a_2, a_3, b_1\}$ is a dominant resolving set and we obtain that $\g_{md}(\Pi)=4$.

Next, we show that $\g_{emd}(\Pi)=4$.
Note that the set $A = \{a_1, a_2, a_3, b_1\}$ is an edge resolving set of $\Pi$.
The representations of all edges of $\Pi$ with respect to the set $A =\{a_1, a_2, a_3, b_1\}$ are given in the Table \ref{tab:er2}.

\begin{table}[htbp!]
		\centering
        \begin{tabular}{|c|c|c|c|c|c|c|c|c|c|}
		\hline
	    Edges &$a_1b_1$ & $a_2b_1$ & $a_3b_1$& $a_4b_1$&$a_1b_2$ & $a_2b_2$ & $a_3b_2$& $a_4b_2$\\ \hline
		${\small r(.,.)}$ & $(0,1,1,0)$& $(1,0,1,0)$& $(1,1,0,0)$& $(1,1,1,0)$& $(0,1,1,1)$& $(1,0,1,1)$ & $(1,1,0,1)$& $(1,1,1,1)$\\ \hline
        \end{tabular}
        \vspace*{0.3cm}
        \caption{The representations of all edges of $\Pi$ with respect to the set $A =\{a_1, a_2, a_3, b_1\}$.}\label{tab:er2}
\end{table}
Table \ref{tab:er2} shows the representations of all the edges of $\Pi$ with respect to $A$ are unique. Since no two edges have the
same representation with respect to $A$,  we deduce that $\B_e(\Pi) \leq 4$. Suppose that $\B_e(\Pi) = 3$ and let
$A$ be the corresponding edge resolving set in $\Pi$. We divide our discussion into the following two cases:

\noindent{\textbf{Case 1:} Assume that $A=\{a_1, a_2, a_3\}$. Then we have $r(a_1b_1| A) = r(a_1b_2| A) = (0,1,1)$,
$r(a_2b_1| A) = r(a_2b_2| A) = (1,0,1)$, $r(a_3b_1| A) = r(a_3b_2| A) = (1,1,0)$ and $r(a_4b_1| A) = r(a_4b_2| A) = (1,1,1)$,
which clearly arise a contradiction.

\noindent{\textbf{Case 2:} Assume that $A=\{a_1, a_2, a_3, b_1\} \setminus \{ a_p \}$ where $p=1,2,3$. Then, $r(a_pb_2| A) = r(a_4b_2| A) = (1,1,1)$,
which is a contradiction to our assumption.

Since $A$ could not be an edge resolving set, we obtain that $\B_e(\Pi) \geq 4$, which then leads to $\B_e(\Pi) = 4$.
The set $T =\{b_1\}$ is a vertex-edge dominating set, as it dominates all the edges but it does not resolve all the edges of $\Pi$.
The set $A =\{a_1, a_2, a_3, b_1\}$ is an edge resolving set, as it resolves all the edges as well as it dominates all the edges.
So, the set $A=\{a_1, a_2, a_3, b_1\}$ is vertex-edge dominant edge resolving set and we obtain that $\g_{emd}(\Pi)=4$.
Thus, we have the following remark:
\begin{remark}\label{remark03}
The non-tree bipartite graph $\Pi$ in Figure \ref{fig:gamma2} satisfies $\g_{md}(\Pi)=\g_{emd}(\Pi)=4$.
\end{remark}
Based on Remarks \ref{remark01}, \ref{remark02} and \ref{remark03}, we conclude that non-tree bipartite
graphs are not comparable with respect to the dominant metric dimension and the vertex-edge dominant edge metric dimension
of graphs.

\section{Concluding remarks and further open problems}\label{conc}
This paper introduces a new parameter related to vertex-edge domination and distances in a graph.
The vertex-edge dominant edge metric dimension $\g_{emd}(\G)$, in fact, combines the theory of vertex-edge domination and
edge resolvability of graphs. The parameter $\g_{emd}(\G)$ has been calculated for classical families such as paths, cycles, complete graphs,
complete bipartite graphs, wheel graphs and fan graphs. Some general results and bounds are provided. The most
significant contribution is that the vertex-edge dominant edge metric dimension $\g_{emd}(\G)$ is not comparable,
in general, with its vertex-version counterpart i.e. the dominant metric dimension $\g_{md}(\G)$. In fact, for a
graph $\G$, all three cases i.e. $\g_{md}(\G)>\g_{emd}(\G)$, $\g_{md}(\G)<\g_{emd}(\G)$, or $\g_{md}(\G)=\g_{emd}(\G)$
can occur.
Upon considering the class of bipartite graphs, we show that
$\g_{emd}(T_n)$ of a tree $T_n$ is always less than or equal to its dominant metric dimension.
However, we show that for non-tree bipartite graphs, the parameter is not comparable just like general graphs.

Based on the results in this paper, we would like to propose the following open questions which arise naturally.

\begin{problem}
Characterize graphs $\G$ satisfying one of the following three cases:
\begin{itemize}
\item[\emph{(i)}] $\g_{md}(\G)>\g_{emd}(\G)$,
\item[\emph{(ii)}] $\g_{md}(\G)=\g_{emd}(\G)$,
\item[\emph{(iii)}] $\g_{md}(\G)<\g_{emd}(\G)$.
\end{itemize}
\end{problem}
Theorem \ref{treescomparison} shows that trees are comparable with respect to their vertex-edge dominant edge
metric dimension. However, it motivates the following problem.
\begin{problem}
Let $T_n$ be a tree on $n\geq2$ vertices. Characterize all trees $T_n$ with $\g_{md}(T_n)=\g_{emd}(T_n)$.
\end{problem}

Although, the dominant metric dimension and the vertex-edge dominant edge metric dimension are not comparable,
in general, there might exist a relationship between these parameters for special classes of graph such trees,
unicyclic or bipartite graphs etc.
\begin{problem}
Is there any relationship between the dominant metric dimension $\g_{md}(\G)$ and $ve-$dominant edge metric
dimension $\g_{emd}(\G)$  for special classes of graph,
unicyclic or bipartite graphs etc.?
\end{problem}

\begin{problem}
Let $\G$ be a graph with order $n$ and size $m$. Let $\delta(\G)$ (resp. $\Delta(\G)$) be the minimum degree (resp. maximum
degree) of $\G$. Derive upper/lower bounds for $\g_{emd}(\G)$ in terms of $n,m,\delta(\G)$ or $\Delta(\G)$?
\end{problem}

\section{Declarations}

\subsection{Funding}
Sakander Hayat is supported by the Higher Education Commission, Pakistan under grant number
20-11682/NRPU/RGM/R\&D/HEC/2020.

\subsection{Competing Interests}
There are no competing interests.

\subsection{Author Contributions}
All authors contributed equally to this paper.

\subsection{Data Availability}
There is no data associated with this paper.


\end{document}